\documentclass[a4paper,11pt]{amsart}

\usepackage{amssymb}
\usepackage{amsmath}

\numberwithin{equation}{section}

\theoremstyle{plain}
\newtheorem{theorem}{Theorem}[section]
\newtheorem{corollary}[theorem]{Corollary}

\newtheorem{lemma}[theorem]{Lemma}
\newtheorem{proposition}[theorem]{Proposition}
\theoremstyle{definition}
\newtheorem{definition}[theorem]{Definition}

\newtheorem{remark}[theorem]{Remark}

\newtheorem{example}[theorem]{Example}

\newcommand{\Z}{\mathbb{Z}}
\newcommand{\Q}{\mathbb{Q}}
\newcommand{\R}{\mathbb{R}}
\newcommand{\C}{\mathbb{C}}
\newcommand{\F}{\mathbb{F}}

\newcommand{\sL}{\mathcal{L}}
\newcommand{\sO}{\mathcal{O}}
\newcommand{\sM}{\mathcal{M}}

\newcommand{\sN}{\mathcal{N}}
\newcommand{\sX}{\mathfrak{X}}

\newcommand{\ord}{\operatorname{ord}}

\newcommand{\Pic}{\operatorname{Pic}}
\newcommand{\Jac}{\operatorname{Jac}}

\newcommand{\Hom}{\operatorname{Hom}}

\newcommand{\Spec}{\operatorname{Spec}}

\newcommand{\coker}{\operatorname{coker}}
\newcommand{\id}{\operatorname{id}}

\newcommand{\Gr}{\operatorname{Gr}}

\newcommand{\alg}{\operatorname{alg}}
\newcommand{\an}{\operatorname{an}}

\newcommand{\red}{\operatorname{red}}

\newcommand{\Supp}{\operatorname{Supp}}

\renewcommand{\int}{\operatorname{int}}

\renewcommand{\tt}{\frac{1}{T}}
\newcommand{\Maya}{\underline{\operatorname{Maya}}}
\newcommand{\M}{\underline{M}}
\newcommand{\Par}{\underline{\operatorname{Par}}}

\begin{document}


\title{Torsion points on Jacobian varieties
\\
via Anderson's $p$-adic soliton theory}



\author{Shinichi Kobayashi \and Takao Yamazaki}
\date{\today}
\address{Mathematical Institute, Tohoku University,
  Aoba, Sendai 980-8578, Japan}
\email{shinichi@math.tohoku.ac.jp}
\email{ytakao@math.tohoku.ac.jp}

\begin{abstract}
Anderson introduced a $p$-adic version of soliton theory.
He then applied it to the Jacobian variety of 
a cyclic quotient of a Fermat curve
and showed that torsion points of certain prime order
lay outside of the theta divisor. 
In this paper, we evolve his theory further.
As an application,
we get a stronger result on the intersection of the 
theta divisor and torsion points on the Jacobian variety
for more general curves. 
New examples are discussed as well.
A key new ingredient is a map connecting
the $p$-adic loop group and the formal group.
\end{abstract}

\keywords{Sato Grassmannian, $p$-adic tau function,
$p$-adic loop group, formal group}
\thanks{The authors are supported by Inamori Foundation.
The second author is supported 
by Grant-in-Aid for Challenging Exploratory Research (24654001) and
Grant-in-Aid for Young Scientists (A) (22684001)}
\maketitle

\section{Introduction}

In \cite{Anderson}, G. Anderson
introduced a striking new machinery
to investigate the Jacobian variety 
and theta divisor of an algebraic curve,
which he called the {\it $p$-adic soliton theory}.
He then applied his theory 
to the Jacobian variety of 
a cyclic quotient of a Fermat curve and
showed that torsion points of certain prime order
lay outside of the theta divisor
(see Thm. \ref{thm:anderson} and Rem. \ref{rem:main1+} below for details). 
His proof can be divided into two parts:
\begin{enumerate}
\item
A description of 
the Jacobian variety and theta divisor
in terms of the ($p$-adic) Sato Grassmannian, tau function
and loop group.
This part is applicable for a more general class of curves.
\item
Explicit construction of
good elements of the Sato Grassmannian and loop group.
This part involves heavy computation 
depending on the defining equation of the curve.
No attempt has been made
(as far as the authors know)
to carry out this part
for other classes of curves, except \cite{MY}. 
\end{enumerate}

The main purpose of the present paper is
to develop a more general and refined theory for (2).
In \S \ref{sect:cp-anderson} below,
we state our result for a cyclic quotient of Fermat curves,
which explains how our theory refines Anderson's result.
Then in \S \ref{subsect:mainresult},
after introducing more notations,
we will state our main result, 
which explains how general our theory is.
Results in \S \ref{sect:cp-anderson} will be deduced
from results in \S \ref{subsect:mainresult} as a special case.
We will see our main result can be applied to 
many new examples in \S \ref{sect:other-ex}.

A key step in our theory is a map 
\eqref{eq:formal-loop} connecting
the $p$-adic loop group and the formal group,
which is constructed using the theory of 
the Hasse-Witt matrix and Artin-Hasse exponential.
We expect this new map could be useful
for other purpose.

\subsection{Cyclic quotient of Fermat curves}\label{sect:cp-anderson}
Let $p$ be a prime and $K$ a finite extension of $\Q_p$.
Let $d$ be an odd prime such that $p \equiv 1 \mod d$
and let $a \in \Z$ be such that $1<a<l$.
We consider the smooth projective model $X$ 
of an affine curve over $K$
defined by 
\begin{equation}\label{eq:curve}
 y^d=x^a(x-1)^{d+1-a}.
\end{equation}
The genus of $X$ is $g=(d-1)/2$.
Let $\infty$ be the unique point on $X$ 
which does not lie above 
the affine curve \eqref{eq:curve}.
The Jacobian variety $\Jac(X)$ and theta divisor $\Theta$ of $X$ 
are
defined as the group of isomorphism classes
of invertible sheaves on $X$ of degree zero
and
\begin{equation}\label{eq:def-of-theta} 
\Theta := \{ \sL \in \Jac(X) ~|~ H^0(X, \sL((g-1)\infty)) \not= 0 \}
\subset \Jac(X)
\end{equation}
respectively.
For any $n \in \Z_{>0}$, 
we denote by $\Jac(X)[p^n]$
the subgroup of $\Jac(X)$ consisting of the elements
of order divisible by $p^n$.
Fix a primitive $d$-th root of unity $\zeta_d \in K^*$
which actually belongs to $\Z_p \subset K$
by the assumption $p \equiv 1 \mod d$.
We define
\begin{equation}\label{eq:def-of-jac-p-primary-torsion}
\Jac(X)[p^n]_1 := \{ \sL \in \Jac(X)[p^n] 
~|~ \delta^*(\sL)=\zeta_d \sL \},
\end{equation}
where $\zeta_d \sL$ makes sense 
since $\Jac(X)[p^n]$ is a $\Z_p$-module.
We have
$\Jac(X)[p^n]_1 \cong \Z/p^n\Z$
if $K$ is sufficiently large
(see Prop. \ref{lem:jactorsion}).
A part of Anderson's result is the following:

\begin{theorem}[Anderson \cite{Anderson}]\label{thm:anderson}
We have 
$\Jac(X)[p]_1 \cap \Theta = \{ 0 \}.$
\end{theorem}

We will generalize this result by showing the following:

\begin{theorem}\label{thm:main1}
For any  $n \in \Z_{>0}$,
we have 
$\Jac(X)[p^n]_1 \cap \Theta = \{ 0 \}.$
\end{theorem}

\begin{remark}\label{rem:grant}
\begin{enumerate}
\item
In Thm. \ref{thm:anderson}, Anderson also proved
a similar statement for certain translates of $\Theta$
(see Rem. \ref{rem:main1+}).
We will generalize this result as well
(see Thm. \ref{thm:main1+}).
A similar remark applies to Thm. \ref{thm:main3} below too.
(see Thm. \ref{thm:stronger-main}).
\item
When $X$ is a hyperelliptic curve
(that happens if and only if $a=2, (d+1)/2, d-1$),
Thm. \ref{thm:main1}
(and the same statement for $p \equiv -1 \mod d$)
was proved by Grant \cite{Grant}.
\end{enumerate}
\end{remark}

\subsection{Main result}\label{subsect:mainresult}
In order to state our main result,
we introduce some notations and definitions.
Let $p$ be a prime and $K$ a finite extension of $\Q_p$.
Let $X$ be a smooth projective geometrically connected curve over $K$
of genus $g \geq 2$.
Suppose that $X$ admits a smooth projective model $\sX$
over the ring of integers $O_K$ in $K$.
We write $Y := \sX \otimes_{O_K} \F$ for the special fiber of $\sX$,
where $\F$ is the residue field of $K$.
Suppose also that we are given 
a distinguished $K$-rational point $\infty \in X(K)$
and  write $\overline{\infty} \in Y(\F)$ for the reduction of $\infty$.
Recall that the 
{\it Weierstrass gap sequence} of $X$ at $\infty$ 
is defined by
\begin{equation}\label{eq:nongap}
\begin{split}
 WG_{\infty}(X) &:= 
\{ n \in \Z_{\geq 0} ~|~ 
H^0(X, \sO_X(n \infty)) =
H^0(X, \sO_X((n-1) \infty)) \},
\end{split}
\end{equation}
which is a subset of $\{ 1, 2, \ldots, 2g-1 \}$
with cardinality $g$.
This definition applies to $Y$ and $\overline{\infty}$ as well.
We define $\Jac(X), \Jac(X)[p^n]$ and $\Theta$ 
as in \S \ref{sect:cp-anderson}
(see \eqref{eq:def-of-theta} and around).
The main results of this paper
are the following two theorems.
\begin{theorem}\label{thm:main3}
Let $\delta : X \to X$ be an automorphism of $X$ defined over $K$ 
such that $\delta(\infty)=\infty$.
We suppose the following conditions:
\begin{enumerate}
\item The order $d$ of $\delta$ satisfies
$d \geq 2g+1$ and $p \equiv 1 \mod d$;
\item $Y$ is ordinary;
\item $WG_{\infty}(X) = WG_{\overline{\infty}}(Y)$.
\end{enumerate}
It then turns out that, for any uniformizer $t$ at $\infty$,
the value $\zeta_d  = (t/\delta^*(t))(\infty)$
of the rational function $t/\delta^*(t)$ at $\infty$
is a primitive $d$-th root of unity
and is independent of the choice of $t$
(see \S \ref{sect:good-triv-with-auto}).
Let $n \in \Z_{>0}$ and define
$\Jac(X)[p^n]_1$ by the formula \eqref{eq:def-of-jac-p-primary-torsion}.
(We have
$\Jac(X)[p^n]_1 \cong \Z/p^n\Z$
if $K$ is sufficiently large, see see Prop. \ref{lem:jactorsion}).
Then we have
\[ \Jac(X)[p^n]_1  \cap \Theta = \{ 0\}. 
\]
\end{theorem}

Let
$\widehat{\Jac}(X)[p]$
be the kernel of the 
specialization map
$\Jac(X)[p] \to\Jac(Y)[p]$
(see \eqref{eq:pic-reduction}).
This is isomorphic to the group of $p$-torsion points
on the formal group $\hat{J}_X/O_K$ attached to
the Jacobian variety of $X$.
If $Y$ is ordinary and if $K$ is sufficiently large, we have 
$\widehat{\Jac}(X)[p] \cong(\Z/p\Z)^g$.

\begin{theorem}\label{thm:main2}
We suppose the following conditions:
\begin{enumerate}
\item $p \geq 2g$;
\item $Y$ is ordinary;
\item $WG_{\infty}(X)=WG_{\overline{\infty}}(Y)=\{1, 2, \ldots, g \}$.
(In other words, both $\infty$ and $\overline{\infty}$ are 
{\it non-Weierstrass points}).
\end{enumerate}
Then we have
\[  |\, \widehat{\Jac}(X)[p] \,\cap\, \Theta \,| 
    \leq p^{g-1}. 
\]
\end{theorem}

In some cases, we expect that 
Theorem \ref{thm:main2} combined with other methods could be used to 
determine the set
$\Jac(X)[p] \cap \Theta$ completely.
(A similar strategy is taken in \cite{C} and \cite{Tamagawa}
for the Manin-Mumford conjecture.)
In this paper, however,
we do not pursuit this direction and 
leave it as a further problem.

\begin{remark}\label{rem:intro-rem-1}
\begin{enumerate}
\item
Theorem \ref{thm:main1} will be
deduced from Theorem \ref{thm:main3}
as a special case.
See Theorem \ref{thm:main1+} for details.
\item
The assumption in Theorem \ref{thm:main3} is quite restrictive.
If $X$ admits an automorphism of {\it prime} order $d >g+1$,
then $X$ must be isomorphic to a cyclic quotient of
a Fermat curve over an algebraic closure of $K$ (see \cite{Homma}).
There are, however, many curves that admits an automorphism
of non-prime order $d \geq 2g+1$.
We will discuss such examples in \S \ref{sect:other-ex}.
\item
On the other hand, the assumption in Theorem \ref{thm:main2} is quite mild.
Indeed, the assumptions (2) and (3) are satisfied in ``generic'' situation.
\end{enumerate}
\end{remark}

\subsection{The $p$-adic soliton theory}
We explain the outline of the proof.
We use the notations 
introduced in \S \ref{subsect:mainresult}.

The completion $\widehat{K(X)}_{\infty}$
of the function field $K(X)$ of $X$ at $\infty$
is isomorphic to 
the field of Laurent power series $K((\tt))$.
We fix such an isomorphism,
and let $A$ be the image of
the composition 
$H^0(X \setminus \{ \infty \}, \sO_X )
\hookrightarrow K(X) \hookrightarrow
\widehat{K(X)}_{\infty} \cong K((\tt))$.
Then $A$ is an element of
the {\it $A$-part of the Sato Grassmannian} 
$\Gr_A^{\alg}(K)$, which is by definition
the set of all $A$-submodules $V$ of $K((\tt))$
such that both 
$V \cap K[[\tt]]$ and $K((\tt))/(V+K[[\tt]])$ 
are finite dimensional over $K$.
A pair $(\sL, \sigma)$ of an invertible sheaf $\sL$ on $X$
and a trivialization $\sigma$ of $\sL$ 
on the infinitesimal neighborhood at $\infty$
is called a {\it Krichever pair}.
For a Krichever pair $(\sL, \sigma)$,
we write $V(\sL, \sigma)$ for the image of the composition
$H^0(X \setminus \{ \infty \}, \sL )
\hookrightarrow K(X) \hookrightarrow
\widehat{K(X)}_{\infty} \cong K((\tt))$,
where the first arrow is induced by $\sigma$.
Then $V(\sL, \sigma) \in \Gr^{\alg}_A(K)$, and
the correspondence $(\sL, \sigma) \mapsto V(\sL, \sigma)$
defines a bijection between
the set of isomorphism classes of Krichever pairs
and $\Gr_A^{\alg}(K)$.
Consequently, we obtain a surjective map 
\[ [ \cdot ]_A : \Gr_A^{\alg}(K) \to \Pic(X), 
\quad
V=V(\sL, \sigma) \mapsto [V]_A :=\sL
\]
where $\Pic(X)$ is the Picard group of $X$.
(So far the base field $K$ can be arbitrary.)

Now let us assume that $A$ has
a nice $p$-adic property 
which we call {\it strictly integral}
(see Definition \ref{def:bounded};
roughly speaking, 
this corresponds to the assumption (3)
in Theorem \ref{thm:main3}).
Then we have
the {\it Sato tau function} $\tau_{A^{\an}} : \Gamma \to K$
which enjoys the `key property' explained below.
Here, $\Gamma$ is the {\it $p$-adic loop group}
which contains a subgroup $\Gamma_+$ 
of all formal power series 
$h(T) \in 1 + T O_K[[T]]$ 
whose radius of convergence is strictly larger than $1$.
There exists a non-trivial action of $\Gamma$
on a subset of $\Gr_A^{\alg}(K)$ (which contains $A$).
The `key property' of the 
Sato tau function is that, for $h(T) \in \Gamma_+$, 
one has $\tau_{A^{\an}}(h(T))=0$ 
if and only if $[h(T) A]_A \in \Theta$. 
Therefore,
given an invertible sheaf $\mathcal{P}$ on $X$,
one can prove $\mathcal{P} \not\in \Theta$
by the following strategy:
\begin{enumerate}
\item
Find a good isomorphism $\widehat{K(X)}_{\infty} \cong K((\tt))$
for which $A$ is strictly integral.
\item
Construct $h(T) \in \Gamma_+$ such that
$[h(T)A]_A = \mathcal{P}$.
\item
Show that $\tau_{A^{\an}}(h(T)) \not= 0$.
\end{enumerate}

When $X$ is given by \eqref{eq:curve},
Anderson worked out the constructions (1) and (2) 
by explicit calculation using the defining equation.
We will develop a general theory for (1) and (2)
by using the Hasse-Witt matrix of $Y$
and Artin-Hasse exponential.
Let us explain a bit more about (2),
under the situation of Theorem \ref{thm:main3}.
In order to construct $p$-torsion points, 
Anderson used the {\it Dwork exponential} \cite[\S 3.5]{Anderson}
\[ h(T) = \exp( (-p)^{1/(p-1)}((uT)-(uT)^p) )  \]
for suitable $u \in O_K^*$.
The key property is that
the radius of convergence of $h(T)$ is $>1$
and hence $h(T)$ belongs to $\Gamma_+$.
In order to construct $p^n$-torsion points for $n \geq 1$, 
we will use the {\it Artin-Hasse exponential}
\[ e^{AH}(T) := 
\exp(\sum_{k=0}^{\infty} \frac{1}{p^k} T^{p^k}).
\]
The key property is that
$e^{AH}(T)$ belongs to $\Z_{(p)}[[T]]$,
and hence
$h(T; \pi) := e^{AH}(\pi T)$ belongs to $\Gamma_+$
for any element $\pi$ of the maximal ideal of $O_K$.
We will construct 
a certain formal power series 
$l_1(X)$
(see \eqref{eq:log-fgl-diagonal2})
which has the following property:
if $l_1(\pi)=0$ and the absolute value of 
$\pi$ is $p^{-1/(p^n-p^{n-1})}$,
then $[h(T; \pi)A]_A$ belongs to $\Jac(X)[p^n]_1$.
The coefficients of $l_1(X)$ are
closely related to 
the {\it Hasse-Witt matrix} of $Y$,
which is invertible if and only if $Y$ is ordinary.
This last fact enables us 
to have a good control of the absolute value of $\pi$
such that $l_1(\pi)=0$.
(Actually, $l_1(X)$ gives rise to the logarithm function of
a direct summand of 
the formal group $\hat{J}_X/O_K$ arising from the Jacobian variety
of $X$, at least when $K$ is absolutely unramified. 
See Theorem \ref{thm:fgl} and Remark \ref{rem:log-of-fgl-diag}.)

For (3), we follow Anderson's method \cite{Anderson}.
Namely, we use an  expression of
$\tau_{A^{\an}}(h(T))$ as an infinite sum of the product of
the {\it Schur functions} and {\it Pl\"ucker coordinates}
(see Theorem \ref{thm:tau} (2)).
This fact was also critical
in the classical (complex analytic) soliton theory
(compare \cite[Proposition 8.3]{SW}).

\subsection{Notations}\label{sect:notations}
For an abelian group $A$ and $n \in \Z$, 
we write $A[n] := \{ a \in A ~|~ na=0 \}$.
For any scheme $S$, 
we denote by $\Pic(S)$ the Picard group of $S$.
If $X$ is a smooth projective irreducible
curve over a field,
we write $\Jac(X) = \{ \sL \in \Pic(X) ~|~ \deg \sL=0 \}$.

\section{Sato Grassmannian and Krichever correspondence}

\subsection{Partitions and Maya diagrams}\label{sect:par-maya}
A sequence $\kappa = (\kappa_i)_{i=1}^{\infty}$ of 
non-negative integers is called a {\it partition} 
if $\kappa_i \geq \kappa_{i+1}$ for all $i \in \Z_{>0}$
and if $\kappa_i=0$ for all sufficiently large $i$.
For a partition $\kappa$,
we define the {\it length} $l(\kappa)$ 
and {\it weight} $|\kappa|$ of $\kappa$ by 
\[ l(\kappa) := 
\min \{ i \in \Z_{>0} ~|~ \kappa_i = 0 \} - 1 \in \Z_{\geq 0}
\qquad
\text{and}
\qquad
|\kappa| := \sum_{i=1}^{\infty} \kappa_i \in \Z_{\geq 0}.
\]
We write $\Par$ for the set of all partitions.
The set $\Par$ is equipped with a partial ordering
defined by $\lambda \leq \kappa$
if and only if $\lambda_i \leq \kappa_i$ for all $i \in \Z_{>0}$.

A subset $M$ of $\Z$ is called a {\it Maya diagram} 
if both $M \cap \Z_{\leq 0}$ and $\Z_{>0} \setminus M$
are finite.
We write $\Maya$ for the set of all Maya diagrams.
For $M \in \Maya$,
we define the {\it index} $i(M)$ 
and {\it partition} $\kappa(M) = (\kappa_i(M))_{i=1}^{\infty}$ 
of $M$ by
$i(M) := |M \cap \Z_{\leq 0}| - |\Z_{>0} \setminus M| \in \Z$
and
$\kappa_i(M) := i-i(M) - s_i(M)$,
where $\{ s_i(M) \}_{i=1}^{\infty}$ 
is a (unique) increasing sequence of integers
such that $M = \{ s_i(M) ~|~ i \in \Z_{>0} \}$.
It is well-known (and easy to show) 
that $\kappa(M)$ is indeed a partition, and that
\begin{equation}\label{eq:maya-par}
   \Maya \to \Z \times \Par,
\qquad M \mapsto (i(M), \kappa(M))
\end{equation}
is a bijective map.

\subsection{Sato Grassmannian}\label{sect:satograss}
Let $F$ be a field.
We work with the field of Laurent power series 
\[ 
F((\tt)) := 
\{ v(T) = \sum_{i=-\infty}^n v_i T^i ~|~ n \in \Z, ~v_i \in F \}
\]
with coefficients in $F$.
The {\it degree} of 
$v(T) = \sum_{i={-\infty}}^{n} v_i T^i  \in F((\tt)) ~(v_i \in F)$ 
is defined to be $\deg v(T):=n$ if $v_n \not= 0$;
if further $v_n=1$ we call $v(T)$ {\it monic}.
For an $F$-linear subspace $V$ of $F((\tt))$,
we define a subset 
\[ \M(V) := \{ \deg v(T) ~|~ v(T) \in V \setminus \{ 0 \} \} \]
of $\Z$
and a natural map
\begin{equation*}
f_V : V\to F((\tt))/F[[\tt]], \qquad  v\mapsto v+ F[[\tt]].
\end{equation*}
Note that we have
\begin{equation}\label{eq:ker-coker}
 \dim \ker(f_V) = |\M(V) \cap \Z_{\leq 0}|,
\qquad
 \dim \coker(f_V) = |\M(V) \setminus  \Z_{> 0}|.
\end{equation}

The {\it Sato Grassmannian} $\Gr^{\alg}(F)$ is the set of 
all $F$-linear subspaces $V$ of $F((\tt))$ such that 
both $\ker(f_V)$ and $\coker(f_V)$ are finite dimensional.
By \eqref{eq:ker-coker},
an $F$-linear subspace $V \subset F((\tt))$
belongs to $\Gr^{\alg}(F)$
if and only if  $\M(V) \in \Maya$.
For $V \in \Gr^{\alg}(F)$,
we call $i(V) := i(\M(V))$ (resp. $\kappa(V) := \kappa(\M(V))$) 
the {\it index} (resp. {\it partition}) of $V$.

Let $V, ~V' \in \Gr^{\alg}(F)$.
We define their product $VV'$ to be 
the $F$-linear subspace of $F((\tt))$
spanned by 
$\{ vv' \in F((\tt)) ~|~ v \in V, v' \in V' \}$,
which belongs to $\Gr^{\alg}(F)$.
We call $V$ and $V'$ {\it homothety equivaliant}
if $V=u(T)V'$ for some $u \in F[[\tt]]^*$.
This equivalence relation is denoted by $\sim$.
If $V \sim V'$, then we have $\M(V)=\M(V')$,
and hence $i(V)=i(V'), ~\kappa(V)=\kappa(V')$.
The product $(V, V') \mapsto VV'$ descends to
the set $\Gr^{\alg}(F)/\sim$
of all homothety equivalence classes.

\subsection{Standard basis and Pl\"ucker coordinate}
\label{sect:nongapsetc}
Let $V \in \Gr^{\alg}(F)$.
Put $M=\M(V)$ 
and write $M = \{ s_i \}_{i=1}^{\infty}$ 
with a strictly increasing sequence
of integers $ \{ s_i \}$.
There exists a unique basis 
$\{ v_i(T) = \sum_j v_{ij}T^j \}_{i=1}^{\infty}$
of $V$
satisfying the following conditions:
\begin{enumerate}
\item $v_i(T)$ is monic of degree $s_i$ for all $i \geq 1$:
\item $v_{i, s_j}=0$ for all $i>j \geq 1$.
\end{enumerate}
We call such $\{ v_i(T) \}_{i=1}^{\infty}$
the {\it standard basis} of $V$.
For a fixed Maya diagram $M = \{ s_i \}_{i=1}^{\infty} \in \Maya$,
the set $\{ V \in \Gr^{\alg}(F) ~|~ \M(V)=M \}$
(called a {\it Schubert cell})
is in one-to-one correspondence with
the set of all
$\{ v_i(T) \in F((\tt)) \}_{i=1}^{\infty}$
satisfying (1) and (2) above.

Let $V \in \Gr^{\alg}(F)$ and
take the standard basis 
$\{ v_i(T) = \sum_j v_{ij}T^j \}_{i=1}^{\infty}$.
The Pl\"ucker coordinate $P_{\lambda}(V)$ 
of $V$ at $\lambda \in \Par$ is defined to be the
common value of
\begin{equation}\label{eq:plucker}
 P_{\lambda}(V) := \det(v_{i, j-\lambda_j-i(V)})_{i, j=1}^n 
\end{equation}
for all sufficiently large $n$.
The following lemma is an immediate consequence of definition
(see \cite[\S 2.6]{Anderson}).

\begin{lemma}\label{lem:plucker}
Let $V \in \Gr^{\alg}(F)$.
We have $P_{\kappa(V)}(V)=1$.
If $\lambda \in \Par$ satisfies $P_{\lambda}(V) \not= 0$,
then we have $\lambda \geq \kappa(V)$.
\end{lemma}

\subsection{$A$-part of the Sato Grassmannian}\label{sect:a-part}
Let $A$ be an $F$-subalgebra of $F((\tt))$
such that
\begin{equation}\label{eq:assumption-on-a}
 A \cap F[[\tt]] = F, \qquad 
g_A := \dim F((\tt))/(A+F[[\tt]]) < \infty. 
\end{equation}
(In the next subsection,
we will construct examples of $A$ arising from geometry.)
It follows that 
$\M(A) \subset \Z_{\geq 0}$
and there exist integers $\mu_1, \ldots, \mu_{g_A}$ such that
\begin{equation}
\label{eq:def-of-mu-sequence}
\Z_{\geq 0} \setminus \M(A) = \{ \mu_1, \ldots, \mu_{g_A} \},
\qquad
1 \leq \mu_1 < \cdots < \mu_{g_A}.
\end{equation}
Hence $A$ is an element of $\Gr^{\alg}(F)$ of index $1-g_A$.
There is a decomposition of $F$-vector spaces
\begin{equation}
\label{eq:decomp0}
   F((\tt)) = A \oplus \tt F[[\tt]] \oplus 
  (\bigoplus_{i=1}^{g_A} F T^{\mu_i}).
\end{equation}

We define the {\it $A$-part of the Sato Grassmannian}
$\Gr^{\alg}_A(F)$
by
\[ \Gr_A^{\alg}(F) := \{ V \in \Gr^{\alg}(F) ~|~ aV \subset V
~\text{for all}~a \in A \}.
\]
Note that $\Gr^{\alg}_A(F)$ is stable 
under the product and homothety equivalence.
The product structure makes $\Gr_A^{\alg}(F)$ 
(resp. $\Gr_A^{\alg}(F)/\sim$)
a commutative semi-group with the unit element $A$
(resp. the class of $A$).

\begin{remark}
The semi-groups $\Gr_A^{\alg}(F)$ and $\Gr_A^{\alg}(F)/\sim$
are not necessary groups in general,
but this is the case in the examples
we will construct in the next subsection.
(See Theorem \ref{thm:krichever} (1) and Remark \ref{rem:sing-cv}
for details.)
\end{remark}

\subsection{Coordinate ring of a curve minus one point}
Let $X$ be a smooth projective geometrically connected curve
of genus $g \geq 2$ over $F$
equipped with an $F$-rational point $\infty \in X(F)$.
Let $\hat{\sO}_{X, \infty}$ be the completion 
of the local ring ${\sO}_{X, \infty}$ of $X$ at $\infty$.
We fix an isomorphism 
$N_0 :\hat{\sO}_{X, \infty} \cong F[[\tt]]$  of $F$-algebras.
We denote
by the same letter $N_0$
the composition map
$\Spec F[[\tt]] \overset{N_0}{\to} 
\Spec \hat{\sO}_{X, \infty} \to X$.
We write $N$ for 
the map $\Spec F((\tt)) \to X$ induced by $N_0$.
We define 
\begin{equation}\label{eq:affinering}
 A =A(X, \infty, N) := \{ N^*(f) \in F((\tt)) ~|~
f \in H^0(X \setminus \{ \infty \}, \sO_X) \}.
\end{equation}
Then $A$ is an $F$-subalgebra of $F((\tt))$
such that $\Spec A \cong X \setminus \{ \infty \}$.
We write $\M(A) = \{ s_i ~|~ i \in \Z_{>0} \}$
with a strictly increasing sequence of integers
$\{ s_i \}_{i=1}^{\infty}$.
The Riemann-Roch theorem shows that
\begin{equation}\label{eq:seq}
s_1=0, \quad s_2 \geq 2 \quad \text{and} 
\quad s_i=i-1+g \quad \text{if}~i > g.
\end{equation}
It follows that $A$ satisfies
\eqref{eq:assumption-on-a}
with $g_A=g$.
The Maya diagram $\M(A)$ of $A$ 
is called the {\it Weierstrass semi-group}
of $X$ at $\infty$.
The complement $\Z_{\geq 0} \setminus \M(A)$
is nothing other than the Weierstrass gap sequence
$WG_{\infty}(X)$ of $X$ at $\infty$ (see \eqref{eq:nongap}).
From \eqref{eq:seq},
we get a constraint on the values
of $\mu_1, \cdots, \mu_g$
defined in \eqref{eq:def-of-mu-sequence}:
\begin{equation}\label{eq:vecmu}
WG_{\infty}(X)=\{ \mu_1, \cdots, \mu_g \},
\qquad 1=\mu_1 < \cdots < \mu_g \leq 2g-1.
\end{equation}

\subsection{Krichever pair}\label{ex:kripair}
An {\it $N$-trivialization} of an invertible sheaf $\sL$ on $X$
is an isomorphism $\sigma : N^*\sL\cong F((\tt))$
of $F((\tt))$-vector spaces
induced by an isomorphism
$\sigma_0 : N_0^* \sL \cong F[[\tt]]$ of $F[[\tt]]$-modules. 
A pair $(\sL, \sigma)$ of
an invertible sheaf $\sL$ on $X$ and an $N$-trivialization $\sigma$ of $\sL$
is called a {\it Krichever pair}.
Two Krichever pairs are said to be isomorphic
if there is an isomorphism of invertible sheaves
compatible with $N$-trivializations.
The set of all isomorphism classes of Krichever pairs
forms an abelian group under tensor product
with the unit element $(\sO_X, N)$.

Suppose  we are given a divisor $D = \sum_{P \in X} n_P P$ on $X$.
The associated invertible sheaf $\sO_X(D)$ admits an $N$-trivialization
$\sigma(D)$ induced by the composition
$\sO_X(D) \hookrightarrow F(X) \overset{N^*}{\to} F((\tt))
\overset{T^{-n_{\infty}}}{\to} F((\tt))$.
(Here $n_{\infty}$ is the coefficient of $\infty$ in $D$.)
Thus we obtain a Krichever pair $(\sO_X(D), \sigma(D))$.

\subsection{Krichever correspondence}
To a Krichever pair $(\sL,\sigma)$,
we associate an $F$-linear subspace $V(\sL, \sigma)$ of $F((\tt))$ by 
\begin{equation}\label{eqn:Krichever_pair}
V(\sL, \sigma) := 
\{\sigma N^* f \in F((\tt)) 
~|~ f \in H^0(X \setminus \{ \infty \}, \sL) \}.
\end{equation}
Note that $A = V(\sO_X, N)$,
and that $V(\sL, \sigma)$ belongs to
the $A$-part of the Sato Grassmannian $\Gr^{\alg}_A(F)$.
If $(\sL, \sigma)$ and $(\sL', \sigma')$ are
two Krichever pairs, one has
$V(\sL \otimes \sL', \sigma \otimes \sigma')
 = V(\sL, \sigma) V(\sL', \sigma')$
(the product introduced at the end of \S \ref{sect:satograss}).

We define the theta divisor $\Theta$ by
\begin{equation}\label{eq:def-of-theta2}
\Theta := \{ \sL \in \Jac(X) ~|~ 
H^0(X, \sL((g-1)\infty)) \not= 0 \}.
\end{equation}

\begin{theorem}[Krichever correspondence]\label{thm:krichever}
\begin{enumerate}
\item
The correspondence $(\sL, \sigma) \mapsto V(\sL, \sigma)$
defines a bijective map, 
compatible with the product structures,
between the set of all isomorphism classes of 
Krichever pairs and $\Gr_A^{\alg}(K)$.
In particular,
the semi-group $\Gr_A^{\alg}(K)$ is actually an abelian group.
Let us denote by 
\begin{equation*}
 [ \cdot ]_A : \Gr_A^{\alg}(F) \to \Pic(X)
\end{equation*}
the canonical surjective homomorphism
characterized by 
$[V(\sL, \sigma)]_A = \sL$
for any Krichever pair $(\sL, \sigma)$.
\item
We have the following:
\begin{enumerate}
\item $\ker [ \cdot ]_A = \{ u(T) A ~|~ u(T) \in F[[\tt]]^* \}$.
In particular, $[ \cdot ]_A$ induces an isomorphism
\begin{equation}\label{eq:krichever-isom}
  (\Gr_A^{\alg}(F)/\sim) ~\cong~ \Pic(X).
\end{equation}
\item $i(V) = \deg[V]_A - g + 1$ for all $V \in \Gr_A^{\alg}(F)$.
\item Let $n \in \Z$ and $V \in \Gr_A^{\alg}(F)$. 
Then, there is an isomorphism
\[ H^0(X, [V]_A (n \infty)) \cong V \cap T^{n} F[[\tt]]. \]
\item
For $V \in \Gr_A^{\alg}(F)$,
it holds
$[V]_A \in \Theta$ 
if and only if $i(V)=1-g$ and $V \cap T^{g-1} F[[\tt]] \not= 0$.
\end{enumerate}
\end{enumerate}
\end{theorem}

A proof can be found in \cite[\S 2.3-2.4]{Anderson}
(see also \cite{Mumford} and \cite[\S 6]{SW}).

Let $\sL \in \Pic(X)$.
It follows from (2c) that
$\M(V(\sL, \sigma)) \in \Maya$ 
does not depend on the choice of
an $N$-trivialization $\sigma$ of $\sL$.
Hence we may write
$\M(\sL) := \M(V(\sL, \sigma))$ and
$\kappa(\sL) := \kappa(V(\sL, \sigma))$.
Suppose now  $\sL \in \Jac(X)$.
Then we have $\M(\sL) \subset \Z_{\geq 0}$.
Recall that the {\it Weierstrass gap sequence}
$WG_{\infty}(\sL)$ of $\sL$ at $\infty$ is defined by
\begin{equation}\label{eq:nongap2}
\begin{split}
 WG_{\infty}(\sL) &:= 
\{ n \in \Z_{\geq 0} ~|~ 
H^0(X, \sL(n \infty)) =
H^0(X, \sL((n-1) \infty)) \}. 
\end{split}
\end{equation}
(We have $WG_{\infty}(X) = WG_{\infty}(\sO_X)$,
see \eqref{eq:nongap}.)
We have
$\M(\sL) = \Z_{\geq 0} \setminus WG_{\infty}(\sL)$.
It holds $\sL \not\in \Theta$ if and only if
$\kappa(\sL)=(0, 0, \cdots)$.
Many results on special divisors
can be stated in terms of $\kappa(\sL)$.
For instance,
Clifford's theorem \cite[Theorem 5.4]{Hartshorne}
can be stated as
an inequality $\kappa(\sL) \leq (g, g-1, \cdots, 1)$.
In particular, we have
\begin{equation}\label{eq:kappa-one-plus-length}
\kappa_1(\sL) + l(\kappa(\sL)) \leq 2g.
\end{equation}

\begin{remark}\label{rem:sing-cv}
The Krichever correspondence can be 
extended to any integral proper geometrically connected
curve $X$ over $F$ 
with a smooth $F$-rational point $\infty \in X(F)$,
upon replacing $\Pic(X)$ by
its compactification \cite{rego}
(which is no longer an abelian group).
See \cite{Mumford} and \cite[\S 6]{SW} for details.
\end{remark}

\subsection{Hermite basis}\label{sect:hermitbasis}
Let ${\boldsymbol \mu} = (\mu_1, \ldots, \mu_g)$ be as in \eqref{eq:vecmu}.
By Serre duality, we have
$WG_{\infty}(X) =
\{ \ord_{\infty}(\omega)+1 ~|~ \omega \in H^0(X, \Omega_{X/F}^1) \}$.
Hence
there exists a unique basis $\omega_1, \cdots, \omega_g$
of $H^0(X, \Omega_{X/F}^1)$ 
such that $N^*(\omega_i) \in \Omega_{F((\tt))/F}^1$
are expanded as
\begin{equation}\label{eq:hermitbasis}
N^*(\omega_i) = \sum_{j=\mu_i}^{\infty} c_{ij} (\tt)^{j-1} d(\tt)
\qquad (i=1, \cdots, g)
\end{equation}
with $c_{ij} \in F$ 
satisfying $c_{i \mu_j} = \delta_{ij}$ (Kronecker's delta)
for all $i, j \in \{ 1, \cdots, g \}$.
Such a basis is called the {\it Hermite basis}.

Let $m \in \Z_{> 0}$.
It follows from \eqref{eq:decomp0}
that there exist (unique)
\begin{equation}\label{eq:decofpower-00-0}
\begin{split}
&\vec{b}^{[m]}(T) =
(b_j^{[m]}(T))_{j=1}^g
\in (A \oplus \tt F[[\tt]])^{\oplus g}
~~\text{ and }~~ 
\mathbf{e}^{[m]} = (e_{i, j}^{[m]})_{i, j=1}^g \in M_g(F)
\\
&\text{such that}\quad
T^{\mu_j m} = 
  b_j^{[m]}(T) + \sum_{i =1}^g e_{i, j}^{[m]}T^{\mu_i}
~\text{ for all}~j \in \{1, \cdots, g \}.
\end{split}
\end{equation}

The following result is due to St\"ohr and Viana:

\begin{proposition}[St\"ohr-Viana \cite{Stohr-Viana}, Proposition 2.3]
\label{thm:hassewitt}
Let $m \in \Z_{> 0}$. 
We have an equality in $M_g(F)$:
\[ (e_{i, j}^{[m]})_{i, j=1}^g = (c_{i, m \mu_j})_{i, j=1}^g. \]
\end{proposition}
\begin{proof}
For the completeness sake,
we recall the proof given in loc. cit. 
We write $b_j^{[m]}(T) = a_j^{[m]}(T) + g_j^{[m]}(T)$
with $a_j^{[m]}(T) \in A$ and $g_j^{[m]}(T) \in \tt F[[\tt]]$.
Then $a_j^{[m]}(T) \omega_i \in \Omega_{F(X)/F}^1$
is regular except at $\infty$,
and its residue at $\infty$ is
given by $c_{i, m \mu_j} - e_{ij}^{[m]}$.
(Here $F(X)$ is the function field of $X$.)
The proposition follows from the residue theorem.
\end{proof}

\subsection{Hasse-Witt invariant}\label{sect:hasse-witt1}
Suppose now that $F$ is a perfect field  of characteristic $p>0$.
We also assume $p \geq 2g$.
There exists an additive map
$C : \Omega_{F(X)/F}^1 \to \Omega_{F(X)/F}^1$
called the {\it Cartier operator}
(see, for example, \cite[A2]{Lang}).
It is characterized by the following properties:
(i) $C(x^p \omega) = x C(\omega)$ for any
$x \in F(X)$ and $\omega \in \Omega_{F(X)/F}^1$;
(ii) $C(x^{p-1}dx) = dx$ for any $x \in F(X)$.
It preserves the space $H^0(X, \Omega_{X/F}^1)$
of regular differentials.

Let $\omega_1, \cdots, \omega_g$
be a basis of $H^0(X, \Omega_{X/F}^1)$.
We take a matrix
$(\gamma_{ij})_{i,j=1}^g \in M_g(F)$ 
such that $C(\omega_i) = \sum_{j=1}^g \gamma_{ij} \omega_j$.
Then
the matrix $(\gamma_{ij}^{1/p})_{i,j=1}^g$ is called 
the {\it Hasse-Witt matrix}
(with respect to the basis $\omega_1, \cdots, \omega_g$),
and its rank is called the {\it Hasse-Witt invariant} of $X$.
The Hasse-Witt invariant is $g$
if and only if $X$ is {\it ordinary}.

Now let us suppose $\omega_1, \cdots, \omega_g$
to be the Hermite basis with expansion \eqref{eq:hermitbasis}.
It follows from the definition that
the Hasse-Witt matrix with respect to $\omega_1, \cdots, \omega_g$
is given by $(c_{i, p \mu_j})$.
More generally, for any $k \in \Z_{\geq 0}$
we have
\[ C^k(\omega_i) = \sum_{j=1}^g c_{i, p^k \mu_j}^{1/p^k} \omega_j, \]
and hence we get an equality in $M_g(F)$ 
\begin{equation}\label{eq:hw-mat-prod}
 (c_{i, p^k \mu_j}) = 
(c_{i, p \mu_j})
(c_{i, p \mu_j}^p)
\cdots
(c_{i, p \mu_j}^{p^{k-1}}).
\end{equation}

\begin{proposition}\label{cor:hassewitt}
Suppose $p \geq 2g$. 
We set 
$\mathbf{e}^{(k)} = \mathbf{e}^{[p^k]}$
for all $k \in \Z_{\geq 0}$
(see \eqref{eq:decofpower-00-0}).
\begin{enumerate}
\item
The matrix $\mathbf{e}^{(1)}$
is the Hasse-Witt matrix of $X$
with respect to the Hermite basis.
\item
The following conditions are equivalent:
\begin{enumerate}
\item
$X$ is ordinary; 
\item
$\det(\mathbf{e}^{(1)}) \in F^*$;
\item
$\det(\mathbf{e}^{(k)}) \in F^*$ for all $k$.
\end{enumerate}
\end{enumerate}
\end{proposition}
\begin{proof}
This is an immediate consequence of 
Proposition \ref{thm:hassewitt}
and \eqref{eq:hw-mat-prod}.
\end{proof}

\begin{remark}
In loc. cit., St\"ohr and Viana also deal with the case $p < 2g$,
but we will not need this result.
\end{remark}

\section{$p$-adic theory of Sato Grassmannian}\label{sect:p-adicsatograss}
In this section, $p$ is a fixed prime number and
$K$ is a finite extension of $\Q_p$.
We write $| \cdot |$ for the absolute value on $K$
such that $|p|=1/p$.
We set $O_{K} = \{ a \in K ~|~ |a| \leq 1 \},~
\sM_{K} = \{ a \in K ~|~ |a| < 1 \}$,
and $\F = O_{K}/\sM_{K}$.

\subsection{The reduction map}\label{sect:padic0}
We study the relation 
between $\Gr^{\alg}(K)$ and $\Gr^{\alg}(\F)$.
For $a = \sum a_i T^i \in K((\tt))$,
we write 
$|| a || := \sup_{i \in \Z} ~
|a_i| \in \R_{\geq 0} \cup \{ \infty \}$.

\begin{definition}\label{def:bounded}
Let $V \in \Gr^{\alg}(K)$.
Let $\{ v_i(T) \}_{i=1}^{\infty}$
be the standard basis of $V$ (see \S \ref{sect:nongapsetc}).
(Note that $||v_i||=1$ if and only if $||v_i|| \leq 1$ because $v_i$ 
is supposed to be monic.)
\begin{enumerate}
\item
We call $V$ {\it bounded} if 
$||v_i||<\infty$ for all $i \in \Z_{>0}$
(hence $||v||<\infty$ for all $v \in V$).
\item
We call $V$ {\it integral} if 
$V$ is bounded and $||v_i|| \leq 1$
for all sufficiently large $i \in \Z_{>0}$.
\item
We call $V$ {\it strictly integral} if 
$V$ is bounded and $||v_i|| \leq 1$
for all $i \in \Z_{>0}$.
\end{enumerate}
We set 
$\Gr^{\alg, \int}(K) := \{V \in \Gr^{\alg}(K) ~|~ V$ is integral $\}$.
\end{definition}

We shall rewrite these properties 
using the {\it reduction map}
\begin{equation}\label{eq:red}
 \red : O_K[[\tt]][T] \to \F((\tt)), \quad
\sum a_i T^i \mapsto \sum (a_i \bmod \sM_K) T^i.
\end{equation}
For a subset $V$ of $K((\tt))$, 
we define
\begin{equation}\label{eq:v-red-def}
\begin{split}
 O(V) &:= \{ v \in V ~|~ || v|| \leq 1 \} =
  V \cap O_K[[\tt]][T],
\\
  V^{\red} &:= \{ \red v \in \F((\tt)) ~|~ v \in O(V) \}.
\end{split}
\end{equation}
The restriction of \eqref{eq:red} induces
a surjection $O(V) \to V^{\red}$.
For $n \in \Z$, we set 
\begin{equation}\label{eq:ven}
 V_n := V \cap T^n K[[\tt]],
\qquad
 (V^{\red})_n := V^{\red} \cap T^n \F[[\tt]].
\end{equation}
The definitions in \ref{def:bounded} are
best understood in terms of the Maya diagrams
of $V$ and $V^{\red}$
(see \S \ref{sect:satograss}).

\begin{proposition}\label{prop:integral}
Let $V$ be a bounded
element of $\Gr^{\alg}(K)$.
\begin{enumerate}
\item 
The following conditions are equivalent:
\begin{enumerate}
\item $V$ is integral;
\item the reduction map induces
surjective maps $O(V_n) \to (V^{\red})_n$
for all sufficiently large $n \in \Z_{>0}$;
\item $\M(V^{\red}) \in \Maya$ and $i(V) = i(V^{\red})$.
\end{enumerate}
If these conditions hold,
we have $\kappa(V) \leq \kappa(V^{\red})$.
\item 
The following conditions are equivalent:
\begin{enumerate}
\item $V$ is strictly integral;
\item the reduction map induces
surjective maps $O(V_n) \to (V^{\red})_n$
for all $n \in \Z_{>0}$;
\item $\M(V)=\M(V^{\red})$;
\item[(c')] $\M(V) \supset \M(V^{\red})$.
\end{enumerate}
\end{enumerate}
\end{proposition}
\begin{proof}
Let 
$\{ v_i(T) \}_{i=1}^{\infty}$ 
the standard basis of $V$,
and let
$\{ s_i \}_{i=1}^{\infty}$ 
(resp. $\{ \bar{s}_i \}_{i=1}^{\infty}$)
be the strictly increasing sequence of integers
such that 
$\M(V) = \{ s_i ~|~ i \in \Z_{>0} \}$
(resp. 
$\M(V^{\red})$ $= \{ \bar{s}_i ~|~ i \in \Z_{>0} \}$).

First we prove (2).
If (a) holds,
then $\{ v_i(T) \}_{i=1}^{\infty}$ is an $O_K$-basis of $O(V)$.
Hence (b) and (c) follow immediately.
It is easy to see
that (c') is implied by either of (b) or (c).
It remains to prove the implication (c') $\Rightarrow$ (a).
We prove its contraposition.
Suppose that $||v_n(T)|| > 1$ for some $n$.
Take $c \in O_K$ such that $|c|=||v_n(T)||^{-1}$.
Then $\deg \red(c v_n(T))$ is an element of 
$\M(V^{\red})$ which does not belong to $\M(V)$.
This completes the proof of (2).

We prove (1).
Suppose (a).
Take $n_1 \in \Z_{>0}$ such that
$||v_i(T)||=1$ for all $i \geq n_1$.
Let $V'$ be the $K$-linear span of $\{ v_i(T) \}_{i=n_1}^{\infty}$.
Then $V'$ is a strictly integral element of $\Gr^{\alg}(K)$.
Thus we get the surjectivity of
$O(V_{n}') \to (V^{' \red})_n$ for all $n$ by (2).
On the other hand, let $U$ be the $K$-linear span of 
$v_1(T), \cdots, v_{n_1-1}(T)$.
As we have remarked before \eqref{eq:ven},
the reduction map $O(U) \to U^{\red}$ is surjective.
It follows that
$O(U_n) \to (U^{\red})_n$ is surjective
for any $n \geq s_{n_1-1}$
because we have $U_n=U$ and $U^{\red}=(U^{\red})_n$.
We conclude 
that $O(V_n) \to (V^{\red})_n$ is surjective
for all $n \geq \max(n_1, s_{n_1-1})$.

Suppose (b).
Take $n_1 \in \Z_{>0}$ such that
$O(V_n) \to (V^{\red})_n$ is surjective for all $n \geq n_1$.
Let $i_1$ be the minimal integer such that $s_{i_1} \geq n_1$.
Let $V'$ be the $K$-linear span of 
$\{ v_i(T) \}_{i=i_1}^{\infty}$.
Then $V'$ is a strictly integral element of $\Gr^{\alg}(K)$,
and we get $\M(V')=\M(V^{' \red})$ by (2).
On the other hand, let $U$ be the $K$-linear span of 
$v_1(T), \cdots, v_{i_1-1}(T)$.
Then $|\M(U^{\red})|=i_1-1$ and
$\M(U^{\red}) \subset (-\infty, s_{i_1}]$.
This proves (c).

Suppose (c).
Take $i_1 \in \Z_{>0}$ such that
$s_i=\bar{s}_i=i-i(V)$ for all $i \geq i_1$.
Let $V'$ be the $K$-linear span of 
$\{ v_i(T) \}_{i=i_1}^{\infty}$.
Then $V'$ is a strictly integral element of $\Gr^{\alg}(K)$ by (2),
and hence $V$ is integral.

If the condition (b) is satisfied,
then $s_i \geq \bar{s}_i$ holds for all $i$.
Hence $\kappa(V) \leq \kappa(V^{\red})$.
\end{proof}

\subsection{$p$-adic analytic Sato Grassmannian}
\label{sect:padicsatograss}
In addition to the field of Laurent power series $K((\tt))$,
we will work with another field
\begin{equation*}
H:=\left\{ \sum_{i=-\infty}^{\infty}a_iT^i \ \biggm|\ 
a_i\in K,\ 
\sup_{i=-\infty}^{\infty}|a_i|<\infty,\ 
\lim_{i\to\infty}|a_i|=0  \right\}.
\end{equation*}
This is a complete discrete valuation field
whose absolute value is given by
\begin{equation*}
\left\|\sum_{i=-\infty}^{\infty} a_i T^i\right\| 
:= \sup_{i=-\infty}^{\infty} |a_i|.
\end{equation*}
We regard $H$ as an ultrametric Banach algebra over $K$ 
equipped with a norm $\| \cdot \|$.
We define closed subspaces $H_+$ and $H_-$ of $H$ by
\[ H_+ := \{ \sum_{i=-\infty}^{\infty} a_iT^i \in H ~|~ a_i=0 ~(i \leq 0) \}, 
\quad
   H_- := \{ \sum_{i=-\infty}^{\infty} a_iT^i \in H ~|~ a_i=0 ~(i > 0) \}
\]
so that we have $H=H_- \oplus H_+$.

Following Anderson \cite[\S 3.1]{Anderson},
we define the {\it $p$-adic Sato Grassmannian} $\Gr^{\an}(K)$ 
to be 
the set of all $K$-linear subspaces $W$ of $H$
such that 
$W$ is the image of an injective $K$-linear map $w : H_+ \to H$
satisfying the following condition:
there exist $i_0 \in \Z$,
a $K$-linear operator $w_1 : H_+ \to H_-$ with $\|w_1 \|\le 1$,
and
a $K$-linear endomorphism $w_2$ on $H_+$ with $\|w_2 \|\le1$
that is a uniform limit of bounded $K$-linear operators of finite rank
(i.e. {\it completely continuous}),
such that $T^{i_0}w : H_+ \to H=  H_- \oplus H_+$
agrees with $w_1 \oplus (1+w_2)$.

This sophisticated definition
admits an elementary interpretation 
 as follows \cite[\S 3.2]{Anderson}.
We regard both $K((\tt))$ and $H$ as
$K$-linear subspaces of $\prod_{i \in \Z} K T^i$
(which itself is no longer a ring).
For a subset $W$ of $H$,
we set $W^{\alg} := W \cap K((\tt))$.
For a subset $V$ of $K((\tt)) \cap H$,
the closure of $V$ in $H$ is denoted by $V^{\an}$.
Note that if $V \in \Gr^{\alg}(K)$ is bounded
(see Definition \ref{def:bounded}),
then $V \subset K((\tt)) \cap H$.

\begin{proposition}[Anderson \cite{Anderson}, \S 3.2]
\begin{enumerate}
\item
If $W \in \Gr^{\an}(K)$, then one has
$W^{\alg} \in \Gr^{\alg}(K)$.
This defines an injective map
$\alg : \Gr^{\an}(K) \to \Gr^{\alg}(K)$.
\item
Let $V \in \Gr^{\alg}(K)$.
The following conditions are equivalent:
\begin{enumerate}
\item There exists $W \in \Gr^{\an}(K)$ such that $V=W^{\alg}$.
\item $V$ is integral (see Definition \ref{def:bounded}).
\end{enumerate}
\end{enumerate}
Consequently, we get a bijective map
\begin{equation}\label{eq:gralg-gran}
\alg : \Gr^{\an}(K) \to \Gr^{\alg, \int}(K)
\end{equation}
whose inverse is given by $V \mapsto V^{\an}$.
\end{proposition}

For a subset $W$ of $H$,
we set $O(W) := \{ w \in W ~|~ ||w|| \leq 1 \}$
and $\sM(W) := \{ w \in W ~|~ ||w|| < 1 \}$.
There is a canonical surjective map
\begin{equation}\label{eq:red2}
 \red : O(H) \to \F((\tt)), \qquad
\sum a_i T^i \mapsto \sum (a_i \bmod \sM_K) T^i
\end{equation}
(which agrees with \eqref{eq:red} on
$O(H) \cap O_K[[\tt]][T]$).

\subsection{The $p$-adic loop group}
For a real number $\rho$ such that $0 \leq \rho < 1$,
we define
\begin{equation}\label{eq:def-of-gamma-rho}
 \Gamma_{\rho} := 
\{ h(T)=\sum_{i=-\infty}^{\infty} h_iT^i \in H^* ~|~
||h(T)|| = |h_0| = 1, ~
|h_i| \leq \rho^i ~\text{for all}~ i > 0 \}.
\end{equation}
(When $\rho=0$, the group $\Gamma_0$ is
written by $\Gamma_-$ in \cite{Anderson}.)
The {\it $p$-adic loop group} $\Gamma$ is 
defined to be the union of $\Gamma_{\rho}$
for all $0 \leq \rho<1$.

We have an action of $\Gamma$ on $\Gr^{\an}(K)$
given by
$h(T) W := \{ h(T) w(T) \in H ~|~ w(T) \in W \}$
for any $h(T) \in \Gamma$ and $W \in \Gr^{\an}(K)$.
This induces an action of $\Gamma$ on $\Gr^{\alg, \int}(K)$
through \eqref{eq:gralg-gran}.
Explicitly, 
for $V \in \Gr^{\alg, \int}(K)$ and $h(T) \in \Gamma$
we have
\begin{equation}\label{eq:action1}
  h(T) V = \{ h(T) w(T) \in H ~|~ w(T) \in V^{\an} \} \cap K((\tt)). 
\end{equation}
The action of 
\[ \Gamma_+ := \{ h(T) = \sum h_i T^i \in \Gamma ~|~
h_0=1, ~ h_i=0 ~\text{for all}~i<0 \}
\]
on $\Gr^{\alg, \int}(K)$ reduces to the trivial action on $\Gr^{\alg}(\F)$,
that is, 
\begin{equation}\label{eq:trivialreductionaction}
(h(T)V)^{\red} = V^{\red}
\qquad \text{for any}~ V \in \Gr^{\alg, \int}(K)
~\text{and}~h(T) \in \Gamma_+.
\end{equation}

\subsection{Schur function}
Let $\lambda$ be a partition (see \S \ref{sect:par-maya})
and $h=\sum h_i T^i \in \Gamma_+$.
The {\it Schur function} $S_{\lambda}(h)$ is defined by
\begin{equation}\label{eq:schur}
S_{\lambda}(h) := \det(h_{\lambda_i -i+j})_{i, j=1}^{l(\lambda)}
\in \Z[h_1, h_2, \cdots].
\end{equation}
If we declare the degree of $h_i$ to be $i$,
then $S_{\lambda}(h)$ is
homogeneous of degree $|\lambda|$.
It follows that, if $h(T) \in \Gamma_{\rho} \cap \Gamma_+$
for some $0<\rho<1$, then we have
\begin{equation}\label{eq:schur-evaluate}
|S_{\lambda}( h(T) )| \leq \rho^{|\lambda|}.
\end{equation}

\subsection{Sato tau function}
Let $W \in \Gr^{\an}(K)$.
Anderson constructed
the {\it Sato tau function} $\tau_W : \Gamma \to K$
which plays a central role in his theory.
We refer to \cite[\S 3.3]{Anderson}
for its definition.
The important properties of the Sato tau function
are summarized in the following theorem:

\begin{theorem}[Anderson \cite{Anderson}, \S 3.3, \S 3.4]\label{thm:tau}
Let $W \in \Gr^{\an}(K)$.
\begin{enumerate}
\item
For $h(T) \in \Gamma$,
we have $\tau_W(h(T))=0$ 
if and only if
\[ h(T) W^{\alg} \cap T^{i(W^{\alg})} K[[\tt]] \not= 0. \]
(Compare Theorem \ref{thm:krichever} (2d).)
\item
Suppose that $W^{\alg}$ is strictly integral
(see Definition \ref{def:bounded}).
Then the following equality (Sato expansion) holds
for any $h(T) \in \Gamma_+$:
\[ \tau_W(h(T)) 
= \sum_{\lambda \in \Par} P_{\lambda}(W^{\alg}) S_{\lambda}(h).
\]
Here $P_{\lambda}(W^{\alg})$ is
the Pl\"ucker coordinate \eqref{eq:plucker}
and $S_{\lambda}(h)$ is the Schur function \eqref{eq:schur}.
\end{enumerate}
\end{theorem}

\begin{corollary}\label{cor:nonvanish}
Let $W \in \Gr^{\an}(K)$ be such that
$W^{\alg}$ is strictly integral.
Let $0< \rho < 1$
and suppose $h(T) \in \Gamma_+ \cap \Gamma_{\rho}$ 
satisfies $|S_{\kappa}(h(T))|=\rho^{|\kappa|}$
with $\kappa := \kappa(W^{\alg})$.
Then we have
$h(T)W^{\alg} \cap T^{i(W^{\alg})} K[[\tt]] = 0.$
\end{corollary}
\begin{proof}
We follow Anderson's proof of \cite[Lemma 3.5.1]{Anderson}.
Since $W^{\alg}$ is strictly integral,
we have $|P_{\lambda}(W^{\alg})| \leq 1$ for all $\lambda$.
Furthermore, Lemma \ref{lem:plucker}
shows that $P_{\kappa}(W^{\alg})=1$ and 
$P_{\lambda}(W^{\alg})=0$ 
unless $\lambda \in \Par$ satisfies $\lambda \geq \kappa$.
On the other hand,
\eqref{eq:schur-evaluate} shows that
if $\lambda \in \Par$ satisfies
$\lambda \not= \kappa$ and $\lambda \geq \kappa$,
then we have
$|S_{\lambda}(h(T))| < |S_{\kappa}(h(T))| = \rho^{|\kappa|}$.
By Theorem \ref{thm:tau} (2),
this proves $\tau_W(h(T)) \not= 0$.
Now Theorem \ref{thm:tau} (1)
completes the proof.
\end{proof}

\subsection{Artin-Hasse exponential}
Put $l(T) := \sum_{k=0}^{\infty} \frac{1}{p^k} T^{p^k}.$
Recall that the Artin-Hasse exponential
$\exp(l(T))$ belongs to $\Z_{(p)}[[T]]$.
Therefore, for any $\pi \in \sM_K$,
the series
\begin{equation}\label{eq:oneparameterloop}
 h(T; \pi) := \exp(l(\pi T)) 
\end{equation}
belongs to $\Gamma_+$.
More precisely,
writing $h(T; \pi) = \sum_{i=0}^{\infty} h_i T^i$
we have
\begin{align}
\label{eq:coeff}
|h_i| \leq |\pi|^i \quad \text{for all}~i \geq 0~
\quad
\text{and}
\quad
h_i = \frac{\pi^i}{i!} \quad \text{for}~ i= 0, 1, \cdots p-1,
\end{align}
whence $h(T; \pi) \in \Gamma_{|\pi|} \cap \Gamma_+$.
For a finite sequence of positive integers
${\boldsymbol \mu}=(\mu_1, \ldots, \mu_g)$
and for $\vec{\pi}=(\pi_i)_{i=1}^g \in \sM_K^{\oplus g}$,
we define
\begin{equation}\label{eq:multi-h}
  h_{{\boldsymbol \mu}}(T; \vec{\pi}) 
  := \prod_{i=1}^g h(T^{\mu_i}; \pi_i) 
   = \exp( \sum_{i=1}^g l(\pi_i T^{\mu_i})) 
   \in \Gamma_{\rho} \cap \Gamma_+,
\end{equation}
where $\rho := \underset{i}{\max}  |\pi_i|^{1/\mu_i}$.

\begin{proposition}\label{prop:nonvanishing}
We suppose one of the following conditions:
\begin{enumerate}
\item
Let $\pi \in \sM_K \setminus \{ 0 \}$ and
set $h(T)=h(T; \pi), ~\rho=|\pi|$. 
Let $\kappa \in \Par$ be a partition satisfying
$\kappa_1+l(\kappa) \leq p$.
\item
Let $g$ be an integer such that $0<g<p$ and
$\vec{\pi}=(\pi_i)_{i=1}^g \in \sM_K^{\oplus g}$.
Suppose $|\pi_i| \leq |\pi_g|$ for all $i=1, 2, \cdots, g-1$.
Set ${\boldsymbol \mu}=(1, 2, \cdots, g), ~
h(T)=h_{{\boldsymbol \mu}}(T; \vec{\pi})$,
$\rho = |\pi_g|^{1/g}$, and 
$\kappa=(g, 0, 0, \cdots)$.
\end{enumerate}
Then we have
$h(T) \in \Gamma_+ \cap \Gamma_{\rho}$ 
and
$|S_{\kappa}(h(T))|=\rho^{|\kappa|}$.
Consequently, 
we have 
$h(T)W^{\alg} \cap T^{i(W^{\alg})} K[[\tt]] =0$
for any $W \in \Gr^{\an}(K)$ such that
$W^{\alg}$ is strictly integral
and $\kappa = \kappa(W^{\alg})$.
\end{proposition}
\begin{proof}
It is clear from the definition 
that $h(T) \in \Gamma_+ \cap \Gamma_{\rho}$
in both cases.
We prove $|S_{\kappa}(h(T))|=\rho^{|\kappa|}$.
The proof of (1) is again the same as 
Anderson's proof of \cite[Lemma 3.5.1]{Anderson}.
The point is that, by the latter part of \eqref{eq:coeff},
one can combinatorially compute $S_{\lambda}(h(T))$ 
for a small partition $\lambda$.
The result is (see loc. cit.)
\[ S_{\lambda}(h(T)) =  
\Big(\prod_{1 \leq i \leq l(\lambda),~ 1 \leq j \leq \lambda_i} 
HL(\lambda; (i,j)) \Big)^{-1}
\pi^{|\lambda|}
\qquad \text{if}~~ l(\lambda)+\lambda_1 \leq p,
\]
where the {\it hook length}
$HL(\lambda; (i,j))$
is by definition the cardinality of the set
\[ \{ (i', j') \in \Z_{>0} \times \Z_{>0}
~|~ (i'=i ~\text{and}~ j \leq j' \leq \lambda_i) ~\text{or}~ 
(i \leq i' ~\text{and}~j' =j \leq \lambda_{i'}) \}. 
\]
By the assumption $\kappa_1+l(\kappa) \leq p$,
we see that $HL(\kappa; (i, j))$ is a $p$-adic unit
for all $1 \leq i \leq l(\kappa), ~1 \leq j \leq \kappa_i$.
This shows (1).
Next, we consider (2).
In this case, 
$S_{\kappa}(h(T))$ is given by
the coefficient of $T^g$ in $h(T)$.
By assumption, we get 
\[ |S_{\kappa}(h(T))| 
  =
\Big|\sum_{j_1+2j_2+\cdots+gj_g=g} 
   \frac{\pi_1^{j_1} \cdots \pi_g^{j_g}}{j_1! \cdots j_g !} \Big|
 = |\pi_g|= \rho^{|\kappa|},
\]
which proves (2).
The last statement follows from Corollary \ref{cor:nonvanish}.
\end{proof}

\subsection{$g$-dimensional family of $p$-adic loops}
\label{sect:constructionofloops}
Let $A$ be a $K$-subalgebra of $K((\tt))$.
Assume that 
$A$ satisfies \eqref{eq:assumption-on-a} and
\begin{equation}\label{eq:aisstrictlyinteg}
\text{$A$ is a strictly integral element of $\Gr^{\alg}(K)$~
(see Definition \ref{def:bounded}).}
\end{equation}
Let $A^{\an}$ be the closure of $A$ in $H$,
so that $A^{\an}$ is a $K$-subalgebra of $H$.
By the definition of 
the action of $\Gamma$ on 
$\Gr^{\an}(K)$ and on $\Gr^{\alg, \int}(K)$
(see \eqref{eq:action1}),
we have
\begin{equation}\label{eq:kernel-of-lpgp-action}
   \{ h(T) \in \Gamma ~|~ h(T)A = A \} =
   \{ h(T) \in \Gamma ~|~ h(T)A^{\an} = A^{\an} \} 
  = (A^{\an})^* \cap \Gamma.
\end{equation}
We define an abelian group $\Gamma_A$ by
\[
   \Gamma_A := \Gamma / ( (A^{\an})^* \cap \Gamma) \Gamma_0.
\]
(Here $\Gamma_0$ is the group
\eqref{eq:def-of-gamma-rho} with $\rho=0$.)
We write $[h(T)]$ for the class of $h(T) \in \Gamma$
in $\Gamma_A$.
Then we get a well-defined injective map
\begin{equation}\label{eq:gamma-to-gr-mod-sim}
 \Gamma_A \to \Gr_A^{\alg, \int}(K)/\sim, \qquad
   [h(T)] \mapsto [h(T)A]_A.
\end{equation}

Recall from \S \ref{sect:a-part} that 
the assumption \eqref{eq:assumption-on-a} implies that
$\M(A) \subset \Z_{\geq 0}$ and
there exists an increasing sequence 
$1 \leq \mu_1 < \cdots < \mu_{g_A}$
of positive integers satisfying \eqref{eq:def-of-mu-sequence}.
Set $g :=g_A = 1-i(\M(A))$ and ${\boldsymbol \mu}:=(\mu_1, \ldots, \mu_g)$.
Using \eqref{eq:multi-h},
we define a map 
\begin{equation}\label{eq:formal-loop}
 \sM_K^{\oplus g} \to \Gamma_A,
 \qquad
 \vec{\pi}=(\pi_i)_{i=1}^g \mapsto  [h_{{\boldsymbol \mu}}(T; \vec{\pi})].
\end{equation}
(It should be noted that \eqref{eq:formal-loop} 
is not a group homomorphism,
even if we endow $\sM_K^{\oplus g}$ with an abelian group structure
induced by the formal group in Theorem \ref{thm:fgl} below.)

\begin{proposition}\label{prop:inj}
The map \eqref{eq:formal-loop} is injective.
\end{proposition}

For the proof, we need an auxiliary lemma.
Let $\pi \in \sM_K$.
There is a homomorphism (see \eqref{eq:def-of-gamma-rho})
\begin{equation}\label{eq:gr-map0}
\Gamma_{|\pi|} \to \F[[T]]^{*}, \qquad
\sum_{i=-\infty}^{\infty} h_i T^i
\mapsto  
\sum_{i=0}^{\infty} (\frac{h_i}{\pi^i} \bmod \sM_K) T^i.
\end{equation}
The kernel of \eqref{eq:gr-map0} is
denoted by $\Gamma_{|\pi|-}$.
Hence we get an injective homomorphism
\begin{equation}\label{eq:gr-map}
\Gamma_{|\pi|}/\Gamma_{|\pi|-} \to \F[[T]]^{*}.
\end{equation}

\begin{lemma}\label{lem:fil}
Let $\pi \in \sM_K$ and set $\rho = |\pi|$.
Take $h(T)=\sum_{j \in \Z} h_j T^j \in A^{\an} \cap \Gamma_{\rho}$.
Let $\sum_{i=0}^{\infty} b_i T^i \in \F[[T]]^*$ 
be the image of $h(T)\Gamma_{\rho-}$ by \eqref{eq:gr-map}.
Then, 
we have $b_{\mu_j} = 0$ for all $j=1, \ldots, g$.
\end{lemma}
\begin{proof}
We take the standard basis $\{ a_i(T) = \sum_j a_{ij}T^j \}_i$ 
of $A$  (see \S \ref{sect:nongapsetc}).
Note that $|a_{ij}| \leq 1$ for all $i, j$
by \eqref{eq:aisstrictlyinteg}.
Since $h(T) \in A^{\an}$, we can write $h(T) = \sum_{i=1}^{\infty} c_i a_i(T)$
with $c_i \in K$.
By the definition of standard basis,
for all $i \geq 1$
we have $c_i = h_{s_i}$ 
and hence $|c_i| = |h_{s_i}| \leq \rho^{s_i}$.
Take $j \in \{ 1, \ldots, g \}$. 
Let $l$ be the smallest integer such that
$\mu_j<s_{l}$.
Then we have
$|h_{\mu_j}| = |\sum_{i=s_l}^{\infty} c_i a_{i \mu_j}| 
 \leq \rho^{s_{l}} < \rho^{\mu_j}$.
This proves the lemma.
\end{proof}

\begin{proof}[Proof of Proposition \ref{prop:inj}]
We take a uniformizer $\pi_K \in \sM_K$.
For $q \in \R_{>0}$,
we define $\bar{\sM}_K^q := \sM_K^q/\sM_K^{q+1}$
(equipped with an abelian group structure
induced by the addition of $\sM_K$)
if $q \in \Z_{>0}$, 
and $\bar{\sM}_K^q := 0$ otherwise.
By passing to the quotient,
\eqref{eq:formal-loop} induces an injective {\it homomorphism}
\[ \bigoplus_{j=1}^g \bar{\sM}_K^{i/\mu_j}
  \to
  \Gamma_{|\pi_K|^i}/\Gamma_{|\pi_K|^i -}
\]
for any $i \in \R_{>0}$.
It suffices to show the the injectivity of the homomorphism 
\begin{equation}\label{eq:grofxi}
\bigoplus_{j=1}^g \bar{\sM}_K^{i/\mu_j}
\to \Gamma_{|\pi_K|^i}/
\Gamma_{|\pi_K|^i -} ((A^{\an})^* \cap \Gamma_{|\pi_K|^i})
\end{equation}
induced by \eqref{eq:formal-loop}
for all $i$.
Fix $i$ and take
$\vec{\pi}=(\pi_j) \in \sM^{\oplus g}_K$.
Suppose $|\pi_j| \leq |\pi_K|^{i/\mu_j}$ for all $j$
and the class of $\vec{\pi}$ belongs to the kernel of \eqref{eq:grofxi}.
We need to show the class of $\pi_j$ in $\bar{\sM}_K^{i/\mu_j}$
is trivial for all $j$.
This amounts to showing 
$b_{\mu_j}=0$ for all $j=1, \ldots, g$,
where $\sum_{k=0}^{\infty} b_k T^k \in \F[[T]]^*$
is the image 
of $h(T)=h_{{\boldsymbol \mu}}(T; \vec{\pi})$ by \eqref{eq:gr-map} 
(with $\rho:=|\pi_K|^i$).
However, this holds for any
$h(T) \in \Gamma_{|\pi_K|^i -} 
((A^{\an})^* \cap \Gamma_{|\pi_K|^i})$
by Lemma \ref{lem:fil}.
This completes the proof of Proposition \ref{prop:inj}.
\end{proof}

\subsection{$p^n$-torsion points}\label{sect:pn-tor-pt}
We keep the assumption that
$A \subset K((\tt))$ is a $K$-subalgebra
satisfying \eqref{eq:assumption-on-a} 
and \eqref{eq:aisstrictlyinteg}.
By the assumption \eqref{eq:aisstrictlyinteg},
the decomposition \eqref{eq:decomp0}
restricts to a decomposition of $O_K$-modules
\begin{equation}\label{eq:decomp}
 O_K[[\tt]][T] = O(A) \oplus \tt O_K[[\tt]] \oplus 
   (\bigoplus_{i=1}^g O_K T^{\mu_i}).
\end{equation}
(Recall that ${\boldsymbol \mu}=(\mu_1, \ldots, \mu_g)$
is a sequence of positive integers 
satisfying \eqref{eq:def-of-mu-sequence}.)

By \eqref{eq:decomp},
for all $k \in \Z_{\geq 0}$
there exist (unique)
\begin{equation}\label{eq:decofpower}
\begin{split}
&\vec{b}^{(k)}(T) =
(b_j^{(k)}(T))_{j=1}^g
\in (O(A) \oplus \tt O_K[[\tt]])^{\oplus g}
~\text{and}~ 
\\
&
\mathbf{e}^{(k)} = (e_{i, j}^{(k)})_{i, j=1}^g \in M_g(O_K)
 ~ \text{such that}
\\
&T^{\mu_j p^k} = 
  b_j^{(k)}(T) + \sum_{i =1}^g e_{i, j}^{(k)}T^{\mu_i}
\qquad \text{for all} ~j \in \{1, \cdots, g \}.
\end{split}
\end{equation}
(This is to say,
with the notation in \eqref{eq:decofpower-00-0},
$\vec{b}^{(k)}(T)=\vec{b}^{[p^k]}(T)$
and
$\mathbf{e}^{(k)} = \mathbf{e}^{[p^k]}$.)
Using the matrix $\mathbf{e}^{(k)}$,
we introduce a vector of formal power series:
\begin{equation}\label{eq:logarithm}
 \vec{l}(X_1, \cdots, X_g) 
:= \sum_{k=0}^{\infty} \frac{1}{p^k} \mathbf{e}^{(k)}
\begin{pmatrix}
X_1^{p^k} \\ \vdots \\ X_g^{p^k}
\end{pmatrix}
\in O_K[[X_1, \cdots, X_g]]^{\oplus g}.
\end{equation}
Note that $\vec{l}(\vec{\pi})$ converges in $K^{\oplus g}$
for any $\vec{\pi} \in \sM_{K}^{\oplus g}$.
For $\vec{\pi} = (\pi_1, \cdots, \pi_g) \in \sM_K^{\oplus g}$,
we write $|| \vec{\pi} || := \max( |\pi_1|, \cdots, |\pi_g|)$.
We define 
for each $n \in \Z_{\geq 0}$ 
\begin{equation}
\label{eq:def-of-t_n}
\begin{split}
T_n &:= \{ \vec{\pi} \in \sM_{K}^{\oplus g} ~|~
\vec{l}(\vec{\pi})=0, ~
|| \vec{\pi} || \leq |p|^{1/(p^n - p^{n-1})} \}.
\end{split}
\end{equation}

\begin{proposition}\label{prop:const}
Let $n \in \Z_{>0}$. 
For any $\vec{\pi} \in T_n$,
we have $[h_{{\boldsymbol \mu}}(T; \vec{\pi})^{p^n}]=1$ in $\Gamma_A$.
(See \eqref{eq:multi-h} for the definition of 
$h_{{\boldsymbol \mu}}(T; \vec{\pi})$.)
Consequently, we get an injective map
\begin{equation}\label{eq:map-tn-to-gamma-a}
 T_n \to \Gamma_A[p^n], \qquad
  \vec{\pi} \mapsto [h_{{\boldsymbol \mu}}(T; \vec{\pi})].
\end{equation}
\end{proposition}

For the proof, we need a lemma,
which will be used in the proof of 
Proposition \ref{prop:const-cyc} again.

\begin{lemma}\label{lem:evaluation}
Let $b(T) \in O(A) \oplus \tt O_K[[\tt]]$ and $n \in \Z_{>0}$.
If two elements $c$ and $\pi$ of $\sM_K$ satisfy
$|c| < |p^{n-1}|$ and $|\pi| \leq |p|^{1/(p^n-p^{n-1})}$, 
then we have
\[
[\exp(c \sum_{k=0}^{\infty} \frac{\pi^{p^k}}{p^{k}} b(T) )] = 1 
\qquad \text{in}~ \Gamma_A.
\]
\end{lemma}
\begin{proof}
For all $k \in \Z_{\geq 0}$, we have
\begin{equation}\label{eq:esti}
 |c\cdot \frac{\pi^{p^k}}{p^k}| 
< |p|^{n-1-k+\frac{p^k}{p^{n}-p^{n-1}}} \leq |p|^\frac{1}{p-1}.
\end{equation}
(The latter equality holds if and only if $k=n, n-1$.)
Since the radius of convergence
of $\exp(T)$ is $|p|^{1/(p-1)}$,
it follows that 
$\exp(c \sum \frac{\pi^{p^k}}{p^k} a(T))
\in (A^{\an})^* \cap \Gamma$ for all $a(T) \in O(A^{\an})$,
and
$\exp(c \sum \frac{\pi^{p^k}}{p^k} g(T))
\in \Gamma_0$ for all $g(T) \in \tt O_K[[\tt]]$.
We are done.
\end{proof}

\begin{proof}[Proof of Proposition \ref{prop:const}]
We calculate
\begin{align*}
h_{{\boldsymbol \mu}}(T; \vec{\pi})^{p^n}  
&= \exp( \sum_{j=1}^g 
  \sum_{k=0}^{\infty} \frac{\pi_j^{p^k}}{p^k} T^{\mu_j p^k})^{p^n} 
\\
&= \exp( p^n \sum_{j=1}^g 
  \sum_{k=0}^{\infty} \frac{\pi_j^{p^k}}{p^k} 
  (\sum_{i=1}^g e_{ij}^{(k)}T^{\mu_i} + b_j^{(k)}(T)) )
\\
&\overset{(*)}{=} \exp( p^n \sum_{j=1}^g 
  \sum_{k=0}^{\infty} \frac{\pi_j^{p^k}}{p^k} 
  b_j^{(k)}(T)),
\end{align*}
where we used $\vec{l}(\vec{\pi})=\vec{0}$ at $(*)$.
Now Lemma \ref{lem:evaluation} proves
$[h_{{\boldsymbol \mu}}(T; \vec{\pi})^{p^n}]=1$ in $\Gamma_A$.
The injectivity of \eqref{eq:map-tn-to-gamma-a}
follows from Proposition \ref{prop:inj}.
\end{proof}

The following proposition will be used 
in the proof of Theorem \ref{thm:main2}:

\begin{proposition}\label{prop:existance-of-p-tor-el}
Suppose that $\det(\mathbf{e}^{(1)}) \in O_K^*$.
Then
there exists a finite extension $K'$ of $K$
such that,
upon replacing $K$ by $K'$,
we have $|T_1| = p^g$ and the set
\begin{equation}\label{eq:p-torsion-elements-in-t1}
\{ \vec{\pi}=(\pi_1, \ldots, \pi_g) \in T_1 ~|~ 
|\pi_g|=|p|^{1/(p-1)} \}
\end{equation}
contains at least $p^{g}-p^{g-1}$ elements.
\end{proposition}

\begin{proof}
We shall use the following classical result
\cite[\S 3, Satz 10]{HW}:
Let $\bar{\F}$ be an algebraic closure of $\F$.
Let $L \in M_g(\F)$,
and let $\rho$ be the rank of the matrix
$L L^{(p)} \cdots L^{(p^{g-1})}$,
where $L^{(p^i)}$ denotes the matrix obtained 
from $L$ by raising all the entries to its $p^i$-th power.
Then we have
\begin{equation}\label{eq:number-of-sol-hasse-witt}
\Big|~
\Big\{ 
\begin{pmatrix}
u_1 \\ \vdots \\ u_g
\end{pmatrix}
\in \bar{\F}^{\oplus g}
~\Big|
~
\begin{pmatrix}
u_1 \\ \vdots \\ u_g
\end{pmatrix}+
L
\begin{pmatrix}
u_1^{p} \\ \vdots \\ u_g^{p}
\end{pmatrix}  = \vec{0} 
\Big\}
~\Big|=p^{\rho}.
\end{equation}
Consequently, for any $L \in GL_g(\F)$ we get
\begin{equation}\label{eq:number-of-sol-in-t1}
\Big| ~
\Big\{ 
\begin{pmatrix}
u_1 \\ \vdots \\ u_g
\end{pmatrix}
\in \bar{\F}^{\oplus g}
~\Big|
~
\begin{pmatrix}
u_1 \\ \vdots \\ u_g
\end{pmatrix}+
L
\begin{pmatrix}
u_1^{p} \\ \vdots \\ u_g^{p}
\end{pmatrix}  = \vec{0},
~ u_g \not= 0
\Big\}
~
\Big|
\geq p^g-p^{g-1}.
\end{equation}

To prove the proposition,
we may suppose that there exists
$\varpi \in K$ such that $\varpi^{p-1}=p$. 
Then we have 
$$ \frac{1}{\varpi} \vec{l}(\varpi  X_1, \cdots, \varpi X_g) 
\equiv  
\begin{pmatrix}
X_1 \\ \vdots \\ X_g
\end{pmatrix}+
\mathbf{e}^{(1)}
\begin{pmatrix}
X_1^{p} \\ \vdots \\ X_g^{p}
\end{pmatrix} \mod \varpi.$$
By 
\eqref{eq:number-of-sol-hasse-witt} and
Hensel's lemma,
after replacing $K$ by its finite unramified extension,
we get $p^g$ elements 
$(u_1, \dots, u_g) \in O_K^{\oplus g}$ 
such that $\vec{l}(\varpi u_1, \cdots, \varpi u_g)=0$,
each of which gives rise to an element 
$(\pi_1, \dots, \pi_g)=(\varpi u_1, \dots, \varpi u_g)$ 
of $T_1$.
Moreover, at least $p^g-p^{g-1}$ elements among them
satisfy $|u_g|=1$ by \eqref{eq:number-of-sol-in-t1},
and each of them gives rise to an element 
$(\pi_1, \dots, \pi_g)=(\varpi u_1, \dots, \varpi u_g)$ 
of the set \eqref{eq:p-torsion-elements-in-t1}.
We are done.
\end{proof}

\subsection{Cyclic group action}
Let $d$ be an integer such that $p \equiv 1 \mod d$
and let $\zeta_d \in \Z_p \subset K$ 
be a primitive $d$-th root of unity.
We define a $K$-algebra automorphism
$\delta : K((\tt)) \to K((\tt))$
by $\delta(\sum_n a_n T^n) = \sum_n a_n (\zeta_d T)^n$.
In this subsection,
$A \subset K((\tt))$ is a $K$-subalgebra
satisfying \eqref{eq:assumption-on-a},
\eqref{eq:aisstrictlyinteg} and
the following condition:
\begin{equation}\label{eq:s-preserved}
\text{$\delta$ restricts to an automorphism of $A$}.
\end{equation}
We also assume
\begin{equation}\label{eq:2g-1leqd}
d \geq \mu_g.
\end{equation}

By \eqref{eq:2g-1leqd},
we have $\mu_i \not\equiv \mu_j \mod d$ for any 
$1 \leq i <j \leq g$.
Since the action of $\delta$ on $K((\tt))$
respects the decomposition \eqref{eq:decomp},
the matrix $\mathbf{e}^{(k)}$ 
introduced in \eqref{eq:decofpower}
must be diagonal.
In particular, we have
\begin{equation}\label{eq:det-of-diag-mat}
 \det(\mathbf{e}^{(k)})
  = e_{11}^{(k)} \cdots e_{gg}^{(k)}.
\end{equation}
We define
\begin{equation}\label{eq:log-fgl-diagonal2}
l_i(X)=\sum_{k=0}^{\infty} 
\frac{e_{ii}^{(k)}}{p^k}X^{p^k} ~(i=1, \ldots, g)
\end{equation}
so that we have
\begin{equation}\label{eq:log-fgl-diagonal}
\vec{l}(X_1, \cdots, X_g) = 
\begin{pmatrix}
l_1(X_1) \\ \vdots \\l_g(X_g))
\end{pmatrix}.
\end{equation}
We also define
for each $i \in \{ 1, \ldots, g \}$
\[
  T_{n, i}=\{ \pi_i \in \sM_K ~|~ l_i(\pi_i)=0,
    ~|\pi_i| \leq |p|^{1/(p^n-p^{n-1})} \}
\]
so that
$T_n = T_{n, 1} \times \cdots \times T_{n, g}$.

\begin{proposition}\label{prop:existance-of-p-tor-el2}
Let $i \in \{ 1, \ldots, g \}$ and $n \in \Z_{>0}$.
Suppose that 
$e_{ii}^{(k)} \in O_K^*$ for all $k \in \Z_{\geq 0}$.
Then
there exists a finite extension $K'$ of $K$
such that,
upon replacing $K$ by $K'$,
\begin{equation}\label{eq:p-power-torsion-elements-in-ti}
|T_{n, i}| = p^{n}.
\end{equation}
Moreover, for any $\pi \in T_{n, i} \setminus \{ 0 \}$
there exists $s \in \{ 1, \ldots, n \}$
such that $|\pi| = |p|^{1/(p^s-p^{s-1})}$.
For any $s \in \{ 1, \ldots, n \}$, 
the cardinality of the set
\begin{equation}\label{eq:set-of-exact-order-pk}
 \{ \pi \in T_{n, i} ~|~ |\pi|=|p|^{1/(p^s-p^{s-1})} \}
\end{equation}
is exactly $p^s - p^{s-1}$.
\end{proposition}
\begin{proof}
It is seen from the derivation
that $l_i(X)$ has no multiple root.
Then the proposition is deduced by
looking at the Newton polygon
(see, for example, \cite[Proposition 2.9]{Goss}).
\end{proof}

There exists a unique continuous 
$K$-algebra automorphism $H \to H$ which 
coincides with $\delta$ on $H \cap K((\tt))$.
We denote this automorphism by the same letter $\delta$.
It induces an automorphism $\Gamma_A$,
which is also denoted by $\delta$.
For $\mu \in \Z/d\Z$
and a $\Z_p$-module $M$ equipped
with a $\Z_p$-linear automorphism $\delta$ of order $d$,
we define
\begin{equation}\label{eq:def-of-mu-part}
M_{\mu} := \{ \alpha \in M ~|~
 \delta \alpha = \zeta_d^{\mu} \alpha \}.
\end{equation}

\begin{proposition}\label{prop:const-cyc}
Let $\pi \in T_{n, i}$
for some $n \in \Z_{>0}$ and $i \in \{ 1, \ldots, g \}$,
and let $h(T) := h(T^{\mu_i}; \pi)$
(see \eqref{eq:oneparameterloop}).
Then
we have $[h(T)] \in \Gamma_A[p^n]_{\mu_i}$.
Consequently, we get an injective map
\begin{equation}\label{eq:map-tn-to-gamma-a-cyclic}
 T_{n, i} \to \Gamma_A[p^n]_{\mu_i}, \qquad
  \pi \mapsto [h(T^{\mu_i}; \pi)].
\end{equation}
\end{proposition}

\begin{proof}
We see $[h(T)] \in \Gamma_A[p^n]$
from Proposition \ref{prop:const}.
In order to show 
$[h(T)] \in \Gamma_A[p^n]_{\mu_i}$,
we take $s \in \Z$ such that $|\zeta_d - s| < |p^n|$.
We need to show $\delta([h(T)]) = [h(T)^{s^{\mu_i}}]$.
Recall that
$h(T)=\exp(l(\pi T^{\mu_i}))$
where
$l(T) = \sum_{k=0}^{\infty} \frac{1}{p^k} T^{p^k}.$
Since  $\delta (l(T)) 
= l(\zeta_d T)=\zeta_d l(T)$,
we get (see \eqref{eq:decomp}),
\begin{align*}
\delta&(h(T))  h(T)^{-s^{\mu_i}}
= \exp((\zeta_d^{\mu_i}-s^{\mu_i}) \sum_{k=0}^{\infty} 
    \frac{\pi^{p^k}}{p^k} T^{\mu_i p^k} )
\\
&= \exp((\zeta_d^{\mu_i}-s^{\mu_i}) 
\sum_{k=0}^{\infty} \frac{\pi^{p^k}}{p^k} 
  (e_{ii}^{(k)} T^{\mu_i} + b_i^{(k)}(T)) )
\\
&\overset{(*)}{=} \exp((\zeta_d^{\mu_i}-s^{\mu_i}) 
\sum_{k=0}^{\infty} \frac{\pi^{p^k}}{p^k} b_i^{(k)}(T) ),
\end{align*}
where we used $l_i(\pi)=0$ at $(*)$.
Now Lemma \ref{lem:evaluation} completes the proof.
\end{proof}

\section{Curves over $p$-adic fields}

\subsection{Notations}\label{sec:model}
We use the same notations for
$p, K, | \cdot |, O_K, \sM_K$ and $\F$
as in the previous section.
Let $\pi_K$ be a uniformizer of $O_K$.
Let $\sX$ be a regular scheme over $O_K$
such that
the structure morphism $\sX \to \Spec O_K$
is separated, smooth and projective of relative dimension one
with geometrically connected fibers.
We write $X := \sX \otimes_{O_K} K$
(resp. $Y := \sX \otimes_{O_K} \F$)
for  the generic (resp. closed) fiber.
Suppose that the genus $g$ of $X$
satisfies $g \geq 2$.
Suppose also that there exists 
a section $\widetilde{\infty} : \Spec O_K \to \sX$ 
of the structure morphism.
We regard the image of $\widetilde{\infty}$
as a reduced closed subscheme of $\sX$,
which is also denoted by $\widetilde{\infty}$.
We write $\infty = \widetilde{\infty} \otimes_{O_K} K$ 
(resp. $\overline{\infty}= \widetilde{\infty} \otimes_{O_K} \F$) 
for the generic (resp. closed) point of $\widetilde{\infty}$.

Recall that the inclusion map
$j : X \to \sX$ (resp. $i: Y \to \sX$)
induces 
an isomorphism $j^* : \Pic(\sX) \cong \Pic(X)$
(resp. a surjection $i^* : \Pic(\sX) \to \Pic(Y)$).
The composition 
\begin{equation}\label{eq:pic-reduction}
 \Pic(X) \overset{j^{* -1}}{\cong} \Pic(\sX)
\overset{i^*}{\to} \Pic(Y),
\qquad
\sL \mapsto \bar{\sL} :=i^* j^{*-1} \sL
\end{equation}
is called the {\it specialization}.
The kernel of this map is denoted by
$\widehat{\Jac}(X)$:
\begin{equation}\label{eq:jac-hat}
 \widehat{\Jac}(X) := 
\{ \sL \in \Pic(X) ~|~ \bar{\sL}= 0 \}.
\end{equation}

\subsection{Existance of a good trivialization}
Let $\hat{\sO}_{\sX, \overline{\infty}}$
be the completion of the local ring 
$\sO_{\sX, \overline{\infty}}$
of $\sX$ at $\overline{\infty}$,
which is a regular two dimensional local $O_K$-algebra.
The prime ideal $\mathcal{P}_{\widetilde{\infty}}$ of
$\hat{\sO}_{\sX, \overline{\infty}}$ 
defined by ${\widetilde{\infty}}$
is of height one, hence principal.
Thus there is an isomorphism 
\begin{equation}\label{eq:def-of-sn0}
 \sN_0 :  \hat{\sO}_{\sX, \overline{\infty}} \cong O_K[[\tt]] 
\end{equation}
of $O_K$-algebras such that
$\sN_0^{-1}((\tt)) = \mathcal{P}_{\widetilde{\infty}}$.
Note that we have
$O_K[[\tt]] 
   \subset K((\tt)) \cap H$
and for any $f \in \hat{\sO}_{\sX, \overline{\infty}}$
\begin{equation}\label{eq:val-deg}
 \deg \sN_0(f) = - v_{\widetilde{\infty}}(f), \qquad
   || \sN_0(f) || = | \pi_K |^{v_Y(f)},
\end{equation}
where $v_{\widetilde{\infty}}$ (resp. $v_Y$)
is the normalized discrete valuation on 
$\hat{\sO}_{\sX, \overline{\infty}}$ 
given by 
$\mathcal{P}_{\widetilde{\infty}}$
(resp. the height one prime ideal defined by $Y$).
We write $\sN$ for the induced
embedding of the function field
$K(X)$ of $X$ into the fraction field of $O_K[[\tt]]$.
The isomorphism $\sN_0$ also induces four maps:
\begin{align*}
&N_0 : \hat{\sO}_{X, \infty} \cong K[[\tt]],&
&\bar{N}_0 : \hat{\sO}_{Y, \overline{\infty}} \cong \F[[\tt]],&
\\
&N : \Spec K((\tt)) \to X,&
&\bar{N} : \Spec \F((\tt)) \to Y.&
\end{align*}
We denote 
the composition map
$\Spec O_K[[\tt]] \overset{\sN_0}{\to} 
\Spec \hat{\sO}_{\sX, \bar{\infty}} \to \sX$
(resp.
$\Spec K[[\tt]] \overset{N_0}{\to} 
\Spec \hat{\sO}_{X, {\infty}} \to X$,
resp.
$\Spec \F[[\tt]] \overset{\bar{N}_0}{\to} 
\Spec \hat{\sO}_{Y, \bar{\infty}} \to Y$)
by the same letter $\sN_0$
(resp. $N_0$, resp. $\bar{N}_0$).

Using $N$ and $\bar{N}$,
we define 
(see \eqref{eq:affinering}).
\begin{equation}\label{eq:def-of-a-and-b}
A=A(X, \infty, N) \subset K((\tt)), \qquad
B=A(Y, \overline{\infty}, \bar{N}) \subset \F((\tt)).
\end{equation}

\begin{proposition}\label{prop:a-is-integral}
\begin{enumerate}
\item
The element $A$ of $\Gr^{\alg}(K)$ is integral
(see Definition \ref{def:bounded} (2)).
\item
If $WG_{\infty}(X)=WG_{\overline{\infty}}(Y)$,
then $A$ is strictly integral
(see Definition \ref{def:bounded} (3)).
\end{enumerate}
\end{proposition}

We will prove a more general proposition.
Define a subset of $X$ by
\begin{equation}\label{eq:ustar}
D_* := \{ P ~|~ \text{$P$ is a closed point of
$X$ whose reduction is $\overline{\infty}$} \}
\setminus  \{ \infty \}.
\end{equation}
Proposition \ref{prop:a-is-integral}
is a special case of the following proposition
applied to $D=0$:

\begin{proposition}\label{prop:trivialization}
Let $D$ be a divisor on $X$
such that 
$\Supp(D) \cap D_* = \emptyset$.
Let $(\sL, \sigma) := 
(\sO_X(D), \sigma(D))$ be
the Krichever pair constructed in \S \ref{ex:kripair}.
Then we have the following.
\begin{enumerate}
\item
The element 
$V(\sL, \sigma)$ of $\Gr^{\alg}(K)$ is integral.
\item
If moreover $\M(\sL)=\M(\bar{\sL})$
(see \eqref{eq:pic-reduction}),
then $V(\sL, \sigma)$ 
is strictly integral.
\end{enumerate}
\end{proposition}

For the proof, we need a lemma.

\begin{lemma}
Let $f \in K(X)$ be a rational function on $X$
that has no pole on $D_*$.
Then, we have $\sN(f) \in O_K[[\tt]][\frac{1}{\pi_K}, T]$.
In particular, we have $|| \sN(f) || < \infty$.
\end{lemma}
\begin{proof}
Since $f$ can be written as a ratio of two elements
of $\sO_{\sX, \overline{\infty}} 
(\subset \hat{\sO}_{\sX, \overline{\infty}})$,
Weierstrass preparation theorem shows that
there are two distinguished polynomials
$P(\tt), Q(\tt) \in O_K[\tt]$,
$n \in \Z$, 
and $u(T) \in O_K[[\tt]]^*$
such that $\sN(f) = \pi_K^n P(\tt) Q(\tt)^{-1} u(T)$.
Since a zero of a distinguished polynomial
has absolute value $<1$, 
the assumption implies that $Q$ can be chosen so that
it has no zero other than $\tt=0$,
that is, $Q(\tt)=T^m$ for some $m \in \Z$.
\end{proof}

\begin{proof}[Proof of Proposition \ref{prop:trivialization}]
The lemma proves that $V:=V(\sL, \sigma)$ is bounded.
With the notations in \eqref{eq:pic-reduction},
we set $\tilde{\sL} = j^{* -1}(\sL)$ 
and $\bar{\sL} = i^* \tilde{\sL}$.
Recall that the $N$-trivialization 
$\sigma : N^* \sL \cong K((\tt))$ comes from
an isomorphism $\sigma_0 : N_0^* \sL \cong K[[\tt]]$.
By the assumption 
$\Supp(D) \cap D_*=\emptyset$,
there exists a unique isomorphism
$\tilde{\sigma}_0$ which fits into a commutative diagram
\[
\begin{matrix}
N_0^* {\sL} 
&\overset{{\sigma}_0}{\cong} 
&K[[\tt]]
\\
\cup & & \cup
\\
\sN_0^* \tilde{\sL} 
&\overset{\tilde{\sigma}_0}{\cong} 
&O_K[[\tt]].
\end{matrix}
\]
By restricting $\tilde{\sigma}_0$ to $Y$, 
we get an isomorphism
$\bar{\sigma}_0 : \bar{N}_0^* \bar{\sL} \cong \F[[\tt]]$,
which induces
an $\bar{N}$-trivialization $\bar{\sigma}$  of $\bar{\sL}$.
By construction,
we have $V(\bar{\sL}, \bar{\sigma})=V^{\red}$
(see \eqref{eq:v-red-def}).
Since the degree of an invertible sheaf is
preserved by the specialization  \eqref{eq:pic-reduction},
$\M(V^{\red})$ is a Maya diagram having 
the same index with $\M(V)$.
Now (1) follows from Proposition \ref{prop:integral} (1).
(2) follows from (1) and Proposition \ref{prop:integral} (2).
\end{proof}

\subsection{$p^n$-torsion points of the Jacobian}
From now on we suppose $WG_{\infty}(X)=WG_{\overline{\infty}}(Y)$.
Then $A$ (which we defined in \eqref{eq:def-of-a-and-b}) 
is strictly integral by Proposition \ref{prop:a-is-integral},
hence we can apply the results of \S \ref{sect:pn-tor-pt}.
Let ${\boldsymbol \mu}=(\mu_1, \ldots, \mu_g)$
be as in \eqref{eq:vecmu}.
Let $n \in \Z_{>0}$.
In view of \eqref{eq:trivialreductionaction},
the composition of \eqref{eq:map-tn-to-gamma-a},
\eqref{eq:gamma-to-gr-mod-sim} and \eqref{eq:krichever-isom}
defines an injective map
(see \eqref{eq:jac-hat})
\begin{equation}\label{eq:map-tn-to-jac-tor}
T_n \to \widehat{\Jac}(X)[p^n] 
\qquad
  \vec{\pi} \mapsto [h_{{\boldsymbol \mu}}(T; \vec{\pi})A]_A.
\end{equation}
(See \eqref{eq:multi-h} for the definition of 
$h_{{\boldsymbol \mu}}(T; \vec{\pi})$.)

\begin{proposition}\label{cor:torsion}
Suppose that $Y$ is ordinary and that $p \geq 2g$.
Then there exists a finite extension $K'$ of $K$
such that,
upon replacing $K$ by $K'$,
we have $|T_1| = p^g$ and
\eqref{eq:map-tn-to-jac-tor}
is bijective for $n=1$.
\end{proposition}

\begin{proof}
Let $\mathbf{e}^{(k)} \in M_g(O_K)$
be the matrix introduced in \eqref{eq:decofpower}.
Since $Y$ is ordinary and $p \geq 2g$,
Proposition \ref{cor:hassewitt} shows that
$\det(\mathbf{e}^{(1)}) \in O_K^*$.
By Proposition \ref{prop:existance-of-p-tor-el},
it holds that $|T_1|=p^g$ if we replace $K$ by
its finite extension.
On the other hand,
we have $|\widehat{\Jac}(X)[p]| \leq p^g$
because $Y$ is ordinary.
Since \eqref{eq:map-tn-to-jac-tor} is injective,
it is surjective as well.
\end{proof}

\subsection{Proof of Theorem \ref{thm:main2}}
The statement of the theorem is not affected by
the base change of the base field $K$.
By Proposition \ref{prop:existance-of-p-tor-el},
we may assume that the set
\eqref{eq:p-torsion-elements-in-t1}
contains at least $p^{g}-p^{g-1}$ elements.
Let $\vec{\pi}$ be an element of
\eqref{eq:p-torsion-elements-in-t1}
and set $h(T):=h_{{\boldsymbol \mu}}(T; \vec{\pi})$.
In view of Proposition \ref{cor:torsion},
it suffices to show
$[h(T)A]_A \not\in \Theta$.

The assumption (3) implies that 
${\boldsymbol \mu}=(1, 2, \ldots, g)$ and $\kappa(A)=(g, 0, 0, \ldots)$.
By Proposition \ref{prop:nonvanishing},
we get $h(T)A \cap T^{g-1} K[[\tt]] = 0.$
Now Theorem \ref{thm:krichever} (2d) shows
$[h(T)A]_A \not\in \Theta$.
This completes the proof of Theorem \ref{thm:main2}.
\qed

\subsection{Formal group}
The main result of this subsection is Theorem \ref{thm:fgl} below.
This result will not be used in the proof
of our main results,
but it might be interesting for its own sake.

Let $J_X/K$ be the Jacobian variety of $X$
so that $J_X(K) \cong \Jac(X)$.
Let $\mathcal{J}_X/O_K$ (resp. $\hat{J}_X/O_K$) be the 
N\'eron model (resp. the formal group) of $J_X$.
The group of $O_K$-rational points $\hat{J}_X(O_K)$
on $\hat{J}_X$ is 
naturally identified with $\widehat{\Jac}(X)$
(see \eqref{eq:jac-hat}).

\begin{theorem}\label{thm:fgl}
Suppose $p \geq 2g, ~WG_{\infty}(X)=WG_{\overline{\infty}}(Y)$
and
that $K$ is absolutely unramified.
Then the vector $\vec{l}(X_1, \cdots, X_g)$
defined in \eqref{eq:logarithm}
gives a logarithm function of
a formal group over $O_K$
which is isomorphic to $\hat{J}_X$.
\end{theorem}


\begin{proof}
We recall basic facts in  \cite{Ka}.  
Let $(V, 0)$ be a pointed  formal Lie variety over $O_K$. 
The coordinate ring $A(V)$ of the formal Lie variety $V$ is just 
the formal power series ring  $O_K[[X_1, \dots, X_n]]$  and 
$0$ is a distinguished $O_K$-valued point of $V$. 
We also denote by $V_K := V \otimes_{O_K} K$ the 
formal Lie variety over $K$ obtained by 
the extension of scalars. The coordinate ring 
$A(V_K)$  is by definition
the formal power series ring  
$K[[X_1, \dots, X_n]]$. 
For example, the formal completion
$\widehat{\sX}$  at the closed point $\overline{\infty}$ of 
 ${\sX}$ and $\hat{J}_X$  
  are such pointed varieties. 
 The de Rham cohomology $H^i_{\mathrm{dR}}(V/O_K)$  is 
 defined as the $O_K$-module obtained by taking 
 the $i$-th cohomology groups of 
 the formal de Rham complex of $V/O_K$. Then 
 by the formal Poincar\'e lemma, we have canonically 
 $$
 H^1_{\mathrm{dR}}(V/O_K)=
\{f \in A(V_K)_0 \,|\, \text{$df$ is integral}\}/A(V)_0,
 $$
where we set
\[  A(V)_0 := \{ f \in A(V) ~|~ f(0)=0 \}, 
 \quad
    A(V_K)_0 := \{ f \in A(V_K) ~|~ f(0)=0 \}.
\]
 
 Suppose that $(V, 0)$ has a structure of 
 commutative formal Lie groups with the identity $0$.  
 Let $m, p_1, p_2$ be  the sum and projections 
 $V \times V \rightarrow V$.  An element $a$ of 
$H^1_{\mathrm{dR}}(V/O_K)$ 
 is called primitive  if it satisfies 
 $$
 m^*a=p_1^*a+p_2^*a \qquad \text{in} 
\quad H^1_{\mathrm{dR}}(V\times V/O_K). 
 $$
 We define the Dieudonn\'e module $\mathbb{D}(V/O_K)$ of $V/O_K$ as 
 the subgroup of  $H^1_{\mathrm{dR}}(V/O_K)$ consisting of 
 primitive elements. Explicitly, we have
 $$
 \mathbb{D}(V/O_K)=
\{f \in A(V_K)_0 \,|\, 
\text{$df$, $f(X[+]Y)-f(X)-f(Y)$ are integral}
\}/A(V)_0
 $$
 where $[+]$ denotes the addition of $V$. 
 
 To use Honda's theory of commutative formal groups, 
we also need a variant. 
 We let 
 $$
 H^1_{\mathrm{dR}}(V/O_K; (p))=
\{f \in A(V_K)_0\,|\, \text{$df$ is integral}\}/pA(V)_0
 $$
 and define $\mathbb{D}(V/O_K; (p))$ to be 
 $$
\{f \in A(V_K)_0 \,|\, \text{$df$ is integral, 
$f(X[+]Y)-f(X)-f(Y) \in pO_K[[X, Y]]$} \}/pA(V)_0.
 $$
 The Frobenius map 
\[ F: K[[X_1, \ldots, X_g]] 
\rightarrow K[[X_1, \ldots, X_g]], 
\quad
 f(X_1, \ldots, X_g) \mapsto 
f^{\varphi}(X_1^p, \ldots, X_g^p)
\] 
induces 
  the Frobenius action $F$ on these cohomology groups 
  where $\varphi$ is the Frobenius of $\mathrm{Gal}(K/\Q_p)$. 
  (Actually, one can replace $\varphi$ by
any $\Z_p$-algebra automorphism $\varphi_1$ of $K$
satisfying $x^{\varphi_1} \equiv x \mod p$ for all $x \in O_K$.) 
  Then $p^{-1}F$ induces $\varphi$-linear isomorphisms 
 $$ 
 H^1_{\mathrm{dR}}(V/O_K; (p)) \cong H^1_{\mathrm{dR}}(V/O_K), \quad 
 \mathbb{D}(V/O_K, (p)) \cong \mathbb{D}(V/O_K). 
 $$
 We also get isomorphisms 
 $$ 
 H^1_{\mathrm{dR}}(V/O_K; (p))\otimes K \cong H^1_{\mathrm{dR}}(V/O_K)\otimes K, \quad 
 \mathbb{D}(V/O_K; (p))\otimes K \cong \mathbb{D}(V/O_K)\otimes K. 
 $$
Let $\underline{\omega}_{V/O_K} $ be 
the space of invariant differentials 
of $V/O_K$. 
Then we have a canonical map 
$\underline{\omega}_{V/O_K} \rightarrow \mathbb{D}(V/O_K, (p))$
given by integration.
If there is no non-trivial $O_K$-group homomorphism
from $V$ to $\mathbb{G}_a$,
then this map is injective 
and we may regard  
$\underline{\omega}_{V/O_K}$ as 
a subspace of $\mathbb{D}(V/O_K, (p))$. 
Furthermore, if $V$ is of finite height $h$, 
then  $\mathbb{D}(V/O_K, (p))$ 
is a free $O_K$-module of rank $h$ and 
is generated by 
$\underline{\omega}_{V/O_K}$ as an $O_K[F]$-module.
Here $O_K[F]$ is an $O_K$-algebra
generated by $F$ subject to 
a relation $Fa=a^\varphi F$ for all $a \in O_K$.
(This algebra is non-commutative unless $O_K=\Z_p$).

Now we return to the proof of Theorem \ref{thm:fgl}.
We follow the same argument as in \cite[\S VI]{Ka}. 
Let $\omega_1, \dots, \omega_g$ 
be the Hermite basis of $X$ with expansion \eqref{eq:hermitbasis}.
From \eqref{eq:decomp},
we see that the elements
$e_{ij}^{[m]}$ defined in  \eqref{eq:decofpower-00-0}
belong to $O_K$ for any $m, i, j$.
By Proposition \ref{thm:hassewitt},
we have $c_{ij} \in O_K$ for all $i, j$,
and $\omega_1, \ldots, \omega_g$ form
an $O_K$-basis of $H^0(\sX, \Omega_{\sX/O_K}^1)$.
By the Albanese map $\sX \rightarrow \mathcal{J}_X$ 
defined  by using $\widetilde{\infty}$, 
we have a canonical isomorphism 
$$
H^0(\sX, \Omega_{\sX/O_K}^1) \cong 
H^0(\mathcal{J}_X, \Omega_{\mathcal{J}_X/O_K}^1) 
$$
and regard 
$\omega_1, \dots, \omega_g$ 
as elements of $H^0(\mathcal{J}_X, \Omega_{\mathcal{J}_X/O_K}^1)$. 
Then by the formal completion, we have 
a basis   
$\hat{\omega}_1, \dots, \hat{\omega}_g$ 
of $\underline{\omega}_{\hat{J}_X/O_K}
\subset \mathbb{D}(\hat{J}_X/O_K, (p))$. 
By \cite[Proposition 3.3]{Ho}, there exists an element 
$$
u(F)=p+\sum_{i=1}^\infty B_iF^i \in M_g(O_K)[[F]]
$$
such that 
$$u(F)\begin{pmatrix} 
\hat{\omega}_1 \\
\vdots\\
\hat{\omega}_g
\end{pmatrix}=0 \qquad \text{in} \quad 
\mathbb{D}(\hat{J}_X/O_K; (p)),
$$
which can be seen as an equality 
in $H^1_\mathrm{dR}(\sX/O_K; (p))$ as well. 
We write
$$
\begin{pmatrix} 
\hat{\omega}_1 \\
\vdots\\
\hat{\omega}_g
\end{pmatrix}
=\sum_{j=1}^\infty 
\begin{pmatrix} 
 c_{1j}  \\
\vdots\\
 c_{gj}  
\end{pmatrix}
 (\tt)^{j-1} d(\tt) 
=\sum_{j=1}^\infty \vec{c}_j (\tt)^{j-1} d(\tt).$$
It follows that
$$
u(F)\left(
\sum_{j=1}^\infty \frac{\vec{c}_j}{j} (\tt)^j \right) 
\in pO_K[[\tt]]^{\oplus g}. 
$$
By looking the $n$-th coefficients, we have 
$$
p \frac{\vec{c}_n}{n}+B_1 \frac{(\vec{c}_{n/p})^\varphi}{n/p}
+\cdots +B_k \frac{(\vec{c}_{n/p^k})^{\varphi^k}}{n/p^k} +\cdots 
\in pO_K^{\oplus g} 
$$
where 
we regard as $\vec{c}_m=0$ if $m$ is not an integer. 
Hence we have
$$
u(F)
\left(
\sum_{k=0}^\infty
\frac{1}{p^k}\left(
\vec{c}_{p^{k} \mu_1} 
X_1^{p^k}+\cdots+ \vec{c}_{p^{k} \mu_g} X_g^{p^k}\right)
\right)
\in\; pO_K[[X_1, \cdots X_g]]^{\oplus g}. 
$$
On the other hand, 
the formal power series
$\sum_{k=0}^\infty
\frac{1}{p^k}(
\vec{c}_{p^{k} \mu_1} 
X_1^{p^k}+\cdots+ \vec{c}_{p^{k} \mu_g} X_g^{p^k})
$
coincides with
$\vec{l}(X_1, \cdots, X_g)$ 
by Proposition \ref{thm:hassewitt}.
By our choice of
the Hermite basis, $\mathbf{e}^{(0)}$ is the identity matrix. 
Therefore, the logarithm  of $\hat{J}_X$ and $\vec{l}(X_1, \cdots, X_g)$ have 
the same Honda type with the same linear term. 
By \cite[Theorem 2]{Ho}, 
$\vec{l}(X_1, \cdots, X_g)$ 
is the logarithm of a formal group over $O_K$ that is 
strictly isomorphic to the formal group $\hat{J}_X$. 
\end{proof}

\subsection{Good trivialization with respect to a given automorphism}
\label{sect:good-triv-with-auto}
Assume now that we are given an automorphism
$\delta : X \to X$ over $K$ of order $d$
such that $\delta(\infty)=\infty$.
We also assume $d \geq 2g-1$ and $p \equiv 1 \mod d$.
(From Corollary \ref{cor:torsion3} onward, 
we will assume $d \geq 2g+1$.)
Note that this implies 
$p \geq 2g> \mu_g$ (see \eqref{eq:vecmu}),
and hence both \eqref{eq:s-preserved} and
\eqref{eq:2g-1leqd} are satisfied.
Let $\zeta_d  = (t/\delta^*(t))(\infty) \in K^*$
be the value of the rational function
$t/\delta^*(t)$ at $\infty$,
where $t$ is a uniformizer at $\infty$.
Note that $\zeta_d$ is independent of the choice of $t$.
It follows from the following proposition
that $\zeta_d$ is a primitive $d$-th root of unity.

\begin{proposition}
The isomorphism \eqref{eq:def-of-sn0}
can be chosen so that 
the following diagram is commutative:
\begin{equation}\label{eq:commutative-diag-with-delta}
\begin{matrix}
 \Spec K((\tt)) & \overset{N}{\to} & X \\
    \downarrow & & \downarrow_{\delta} \\
 \Spec K((\tt)) & \overset{N}{\to} & X,
\end{matrix}
\end{equation}
where the left vertical map 
is given by $\sum a_i T^i \mapsto \sum a_i (\zeta_d T)^i$.
\end{proposition}

\begin{proof}
We first take an arbitrary isomorphism \eqref{eq:def-of-sn0}.
Denote by the same letter $\delta$
the automorphism $O_K[[\tt]] \to O_K[[\tt]]$ 
induced by $\delta$ through \eqref{eq:def-of-sn0}.
It suffices to show that
there is a unit $u(T) \in O_K[[\tt]]^*$
such that $\delta(u(T) \cdot \tt) = \zeta_d^{-1} u(T) \cdot \tt$.

Since the ideal $(\tt)$ is preserved by $\delta$,
there exists $v(T) \in O_K[[\tt]]^*$ such that
$\delta(\tt) = v(T) \tt$.
Write $v(T) = \zeta v_1(T)$ with 
$\zeta \in O_K^*$ and $v_1(T) \in 1 + \tt O_K[[\tt]]$.
Since the order of $\delta$ is $d$,
we have 
\begin{equation}\label{eq:h90}
\prod_{i=0}^{d-1} \delta^i(v_1(T)) = 1
\qquad
\text{and}
\qquad \zeta^d=1.
\end{equation}
Define $u(T) = \sum_{i=0}^{d-1} \prod_{j=0}^i \delta^j(v_1(T))$.
Since $u(T) \equiv d \mod (\tt)$,
we have $u(T) \in O_K[[\tt]]^*$.
On the other hand, we have
$v_1(T) \delta(u(T)) = u(T)$ by \eqref{eq:h90}.
Therefore we get
$\delta(u(T) \cdot \tt) 
 = \zeta u(T) \cdot \tt$.
Then we see $\zeta=\zeta_d^{-1}$ by the definition of $\zeta_d$.
We are done.
\end{proof}

From now on, we choose the isomorphism \eqref{eq:def-of-sn0}
in such a way that
the diagram \eqref{eq:commutative-diag-with-delta}
commutes.
By abuse of notation,
we denote by $\delta$
the isomorphism $K((\tt)) \to K((\tt))$
given by $\sum a_i T^i \mapsto \sum a_i (\zeta_d T)^i$.

\begin{remark}\label{rem:log-of-fgl-diag}
Suppose further that
the assumptions of Theorem \ref{thm:fgl} are satisfied.
Then
the formal power series $l_i(X)$ 
defined in \eqref{eq:log-fgl-diagonal2}
is the logarithm of a formal group $\hat{J}_{X, i}$ over $O_K$
for all $i =1, \ldots, g$.
Moreover, we have an isomorphism 
of formal groups over $O_K$
\begin{equation}\label{eq:decomp-of-fgl}
 \hat{J}_X \cong \bigoplus_{i=1}^g \hat{J}_{X, i}. 
\end{equation}
This follows immediately from
Theorem \ref{thm:fgl} and \eqref{eq:log-fgl-diagonal}.
\end{remark}

\subsection{Decomposition of torsion points}
We use the notation introduced in 
\eqref{eq:def-of-mu-part}.

\begin{proposition}\label{lem:jactorsion}
Let $n \in \Z_{>0}$ and $\mu \in \Z/d\Z$.
Suppose that $K$ contains all
$p^n$-torsion points on $\Jac(X)$, 
that is, $|\Jac(X)[p^n]|=p^{2ng}$.
\begin{enumerate}
\item
If $\mu \equiv \mu_i \mod d$ for some $i=1, \cdots, g$, 
then $\widehat{\Jac}(X)[p^n]_{\mu}$ 
is a cyclic group of order $p^n$.
Otherwise, $\widehat{\Jac}(X)[p^n]_{\mu}=0$.
\item
Set $M_{\pm} := \{ \pm \mu_i \mod d ~|~ i=1, \ldots, g \} \subset \Z/d\Z$
and define
\[
\rho := 
\left\{
\begin{array}{ll}
0 & \text{if $\mu \not\in M_+ \cup M_-$}
\\
2 & \text{if $\mu \in M_+ \cap M_-$}
\\
1 & \text{otherwise}.
\end{array}
\right.
\]
Then $\Jac(X)[p^n]_{\mu}$ 
is a free
$\Z/p^n\Z$-module of rank $\rho$.
\end{enumerate}
\end{proposition}
\begin{proof}
Let $T_p\hat{J}_X$ be the Tate module of 
the formal group $\hat{J}_X/O_K$
associated to the Jacobian variety of $X$.
By Tate's theorem \cite{Tate}, we have
\[ T_p\hat{J}_X \otimes_{\Z_p} \C_p \cong 
\Hom_K(H^0(X, \Omega_{X/K}^1), \C_p),
\]
where $\C_p$ is the completion of an algebraic closure of $K$.
On the other hand, it is shown in 
\cite[Theorem 5]{Lewittes}
that
\[ \det(t \cdot \id -\delta | H^0(X, 
  \Omega_{X/K}^1))
= \prod_{i=1}^g (t-\zeta_d^{-\mu_i}).
\]
Combining two results,
we get
$\det(t \cdot \id - \delta | T_p\hat{J}_X)
= \prod_{i=1}^g (t-\zeta_d^{\mu_i})$.
Recall that we assumed $2g-1 \leq d$ so that
$\mu_i \not\equiv \mu_j \mod d$ for any $1 \leq i < j \leq g$.
Now (1) follows
from the isomorphisms
$\widehat{\Jac}(X)[p^n]
\cong T_p\hat{J}_X/p^nT_p\hat{J}_X$.
(2) is a standard consequence of (1).
\end{proof}

\subsection{$\mu$-part of the torsion points}

For $n \in \Z_{>0}$ and $i \in \{ 1, \ldots, g \}$,
the composition of \eqref{eq:map-tn-to-gamma-a-cyclic},
\eqref{eq:gamma-to-gr-mod-sim} and \eqref{eq:krichever-isom}
defines an injective map
\begin{equation}\label{eq:map-tn-to-jac-tor2}
T_{n, i} \to \widehat{\Jac}(X)[p^n]_{\mu_i} 
\qquad
  \pi \mapsto [h(T^{\mu_i}; \pi)A]_A.
\end{equation}
(See \eqref{eq:oneparameterloop} for the definition of 
$h(T; \pi)$.)

\begin{proposition}\label{cor:torsion2}
Let $n \in \Z_{>0}$ and $i \in \{ 1, \ldots, g \}$.
Suppose that $e_{ii}^{(1)} \in O_K^*$.
Then there exists a finite extension $K'$ of $K$
such that,
upon replacing $K$ by $K'$,
we have $|T_{n, i}| = p^n$ and
\eqref{eq:map-tn-to-jac-tor2}
is bijective.
\end{proposition}

\begin{proof}
Recall that 
the matrix $\mathbf{e}^{(k)}=(e_{ij}^{(k)})$ 
introduced in \eqref{eq:decofpower} is diagonal
(see the remark before \eqref{eq:det-of-diag-mat}).
From Proposition \ref{thm:hassewitt} and \eqref{eq:hw-mat-prod},
we have $e_{ii}^{(k)} \in O_K^*$ for all $k$.
By Proposition \ref{prop:existance-of-p-tor-el2}
it holds that $|T_{n, i}|=p^n$ if we replace $K$ by
its finite extension.
On the other hand, we have $|\widehat{\Jac}(X)[p^n]_{\mu_i}| \leq p^n$ 
by Proposition \ref{lem:jactorsion} (1).
Since \eqref{eq:map-tn-to-jac-tor2} is injective,
it is surjective as well.
\end{proof}

\begin{corollary}\label{cor:torsion3}
Suppose $d \geq 2g+1$ and $e_{11}^{(1)} \in O_K^*$.
Let $n \in \Z_{>0}$.
Then there exists a finite extension $K'$ of $K$
such that,
upon replacing $K$ by $K'$,
we have $|T_{n, 1}| = p^n$ and
we have a bijection
\begin{equation}\label{eq:map-tn-to-jac-tor3}
T_{n, 1} \to \Jac(X)[p^n]_1 
\qquad
  \pi \mapsto [h(T; \pi)A]_A.
\end{equation}
\end{corollary}
\begin{proof}
Recall that we have 
$1=\mu_1 < \cdots < \mu_g \leq 2g-1$
(see \eqref{eq:vecmu}).
Hence the assumption $d \geq 2g+1$ implies
that $1=\mu_1 \in M_+ \setminus M_-$
in Proposition \ref{lem:jactorsion} (2).
It follows that
$\Jac(X)[p^n]_1=\widehat{\Jac}(X)[p^n]_1$.
The proposition applied to $i=1$ completes the proof.
\end{proof}

\begin{theorem}\label{thm:stronger-main}
Let $D$ be a divisor on $X$
such that 
$\delta^*(D)=D, ~\deg D = 0$
and 
$\Supp(D) \cap D_*= \emptyset$
(see \eqref{eq:ustar}).
Let $(\sL, \sigma) := 
(\sO_X(D), \sigma(D))$ be
the Krichever pair constructed in \S \ref{ex:kripair}.
Suppose 
$WG_{\infty}(X) = WG_{\overline{\infty}}(Y),$ 
$WG_{\infty}(\sL) = WG_{\overline{\infty}}(\bar{\sL}),$ 
$e_{11}^{(1)} \in O_K^*, ~
p \equiv 1 \mod d$ and $d \geq 2g+1$.
Then we have for any $n \in \Z_{>0}$
\[ \Jac(X)[p^n]_1  \cap (\Theta - \sL) \subset \{ 0\}. 
\]
\end{theorem}

\begin{proof}
By Proposition \ref{prop:trivialization},
$V:=V(\sL, \sigma) \in \Gr_A^{\alg}(K)$ is strictly integral.
By Proposition \ref{prop:existance-of-p-tor-el2},
we may assume $|T_{n, 1}|=p^n$.
Let $\pi \not= 0$ be an element of 
\eqref{eq:set-of-exact-order-pk} 
for some $s \in \{1, \ldots, n\}$
and set $h(T):=h(T; \pi)$.
In view of Corollary \ref{cor:torsion3},
it suffices to show
$[h(T)V]_A \not\in \Theta$.

From \eqref{eq:kappa-one-plus-length}
we see $\kappa_1(V) + l(\kappa(V)) \leq 2g \leq d-1 < p$.
By Proposition \ref{prop:nonvanishing},
we get $h(T)V \cap T^{g-1} K[[\tt]] = 0.$
Now Theorem \ref{thm:krichever} (2d) shows
$[h(T)V]_A \not\in \Theta$.
This completes the proof of Theorem \ref{thm:stronger-main}.
\end{proof}

\subsection{Proof of Theorem \ref{thm:main3}}
From \eqref{eq:det-of-diag-mat} and
Proposition \ref{cor:hassewitt},
we have $e_{11}^{(1)} \in O_K^*$
if $Y$ is ordinary.
Thus Theorem \ref{thm:main3} follows 
from Theorem \ref{thm:stronger-main}
applied to $D=0$.
\qed

\section{Examples}

\subsection{Cyclic quotient of a Fermat curve}\label{sect:fermat}

In this subsection, 
we work under the assumption and notations in
Thm. \ref{thm:main1}. We shall prove the following result,
which implies Thm. \ref{thm:main1} 
as a special case $\sL=0$.

\begin{theorem}\label{thm:main1+}
Let $\sL \in \Jac(X)$ be such that $\delta^*(\sL)=\sL$.
Then, for any  $n \in \Z_{>0}$,
we have 
\[ \Jac(X)[p^n]_1 \cap (\Theta - \sL) \subset \{ 0 \}.\]
\end{theorem}

\begin{remark}\label{rem:main1+}
Anderson \cite{Anderson} proved Thm. \ref{thm:main1+} for $n=1$.
\end{remark}

\begin{proof}
It follows from \cite[Proposition 5.1]{Gonzalez} that
the assumption $p \equiv 1 \mod d$ implies 
that $Y$ is ordinary.
(See also Remark \ref{rem:fermatordinary} below.)
For $i=0, 1$,
let $P_i$ be the closed points on $X$
characterized by $x(P_i)=i, ~y(P_i)=0$,
so that we have $\delta(P_i)=P_i$.

The assumption implies that
there is a $j \in \{ 0, 1, \cdots, d-1 \}$
such that
$\sL \cong \sO_X(j(P_0 - P_1))$ 
(cf. \cite[{\S 4.2}]{Anderson}).
An elementary computation shows
\[ WG_{\infty}(\sL) = WG_{\overline{\infty}}(\bar{\sL}) 
= \{ i \in \{ 0, \ldots, d-1 \} ~|~
\langle \frac{ia+j}{d} \rangle 
+\langle \frac{i(d+1-a)-j}{d} \rangle 
-\langle \frac{i}{d} \rangle = 1 \}
\]
where $\langle \cdot \rangle$ is the fractional part function.
Hence all the assumptions of
Theorem \ref{thm:stronger-main} are satisfied,
and Theorem \ref{thm:main1+} follows.
\end{proof}

In \cite{Anderson},
Anderson computed the decomposition
\eqref{eq:decofpower} explicitly for $k=1$.
His computation can easily be generalized to $k>1$.
Combined with Theorem \ref{thm:fgl},
this result gives rise to an explicit
formula for the formal group of 
the Jacobian variety of $X$.
In the rest of this subsection,
we work out such a formula.

Set $b := d+1-a$.
There exists a unique $u(T) \in 1+\tt \Z[[ T ]]$ such that
\[ u(T)^{b-1} = (u(T)-\frac{1}{T^d})^b. \]
We define 
$x(T) := T^d u(T), ~y(T) := T^{d+1}u(T) = T x(T) \in \Z[[\tt]][T]$
so that 
$x(T) \equiv T^d \mod T^{d-1} \Z[[\tt]]$
and
$y(T) \equiv T^{d+1} \mod T^{d} \Z[[\tt]]$.
One checks $y(T)^d = x(T)^a(x(T)-1)^b$.
We get an embedding of $\Q_p$-algebras
\begin{equation}\label{eq:defofa}
 H^0(X \setminus \{ \infty \}, \sO_X) 
   \hookrightarrow \Q_p((\tt))
\end{equation}
which sends $x$ and $y$ to $x(T)$ and $y(T)$
respectively.
We define the trivialization $N_0$ 
to be the one induced by \eqref{eq:defofa}.

Let $k \in \Z_{\geq 0}$ and write $p^k-1 = dm ~(m \in \Z_{\geq 0})$.
We have
\[ T^{p^k-1}  
= (\frac{y(T)^d}{x(T)^d})^{m} = (\frac{x(T)^a(x(T)-1)^b}{x(T)^d})^m
= x(T)^m(1-\frac{1}{x(T)})^{bm}.
\]
It follows that
\[ T^{p^k} 
= \sum_{i=m-bm}^{m}
\gamma_i^{(k)} x(T)^i T,
\qquad
\gamma_i^{(k)} = 
(-1)^{i-m}
\binom{m b}{m-i} \in \Z.
\]
(Here we define $\binom{0}{0}=1$ by convention.)
If $i>0$, then 
$x(T)^i T = x(T)^{i-1} y(T) \in O(A).$
If $i<0$, then
$x(T)^i T \in \tt \Z_p[[\tt]]$.
This shows that
\begin{equation}\label{eq:explicitdecomp}
 b_1^{(k)}(T) = \sum_{i \not= 0} \gamma_i^{(k)} x(T)^i T,
\qquad
 e_{1, 1}^{(k)} = \gamma_0^{(k)} 
= \binom{\frac{p^k-1}{d} b}{\frac{p^k-1}{d}}
\end{equation}
in \eqref{eq:decofpower}.
Recall that we have a decomposition 
$\hat{J}_X \cong \oplus_{i=1}^g \hat{J}_{X, i}$
of the formal group $\hat{J}_X/\Z_p$ 
associated to the Jacobian variety of $X$
(see \eqref{eq:decomp-of-fgl}).
By Remark \ref{rem:log-of-fgl-diag},
we get the following result
(compare \cite{Ka}).

\begin{proposition}\label{prop:fermatformalgroup}
The formal power series
\[ l_1(X) = 
\sum_{k=0}^{\infty} 
\frac{1}{p^k} \binom{\frac{p^k-1}{d} b}{\frac{p^k-1}{d}} X^{p^k}
\]
is the logarithm function of 
a formal group over $\Z_p$, which is isomorphic  to $\hat{J}_{X, 1}/\Z_p$.
\end{proposition}

\begin{remark}\label{rem:fermatordinary}
In the proof of Theorem \ref{thm:main1},
we made a use of a result in \cite{Gonzalez}.
However, what we actually needed is
$e^{(1)}_{11} \in \Z_{(p)}^*$,
and this follows from \eqref{eq:explicitdecomp}
\end{remark}

\subsection{Other examples}\label{sect:other-ex}

We collect a few more examples for which
Theorem \ref{thm:stronger-main} can be applied.
All results can be proved by the same way
as Theorem \ref{thm:main1} and Proposition \ref{prop:fermatformalgroup},
hence the proofs are omitted.

\begin{example}
(Compare \cite{MY}.)
Let $g \geq 2$ be an integer such that $p \equiv 1 \mod d:=4g$.
Let $X$ be 
the hyperelliptic curve of genus $g$ over $\Q_p$ defined by 
$y^2 = x^{2g+1}+x$.
Let $\infty$ be the unique point on $X$ 
which does not lie above its affine part.
Fix a primitive $d$-th root of unity 
$\zeta_{d} \in \Q_p^*$,
and define an automorphism $\delta$ of $X$ by
$\delta(x, y) = (\zeta_d^{2} x, -\zeta_d y)$. 
The order of $\delta$ is $d$.
Then we have the following for any $n \in \Z_{>0}$:
\begin{enumerate}
\item
It holds
$\Jac(X)[p^n]_1 \cap \Theta  = \{ 0 \}$.
\item
The formal power series
\[ l_1(X) = 
\sum_{k=0}^{\infty} 
\frac{1}{p^k} 
\begin{binom}
{\frac{p^k-1}{2}}
{\frac{p^k-1}{4g}}
\end{binom}
\]
is the logarithm function of 
a formal group over $\Z_p$, which is isomorphic  to $\hat{J}_{X, 1}/\Z_p$.
\end{enumerate}
\end{example}

\begin{example}
Let $l$ be a prime such that $p \equiv 1 \mod d:=l(l+1)$.
Let $X$ be 
the smooth projective model of an affine curve 
over $\Q_p$ defined by 
$y^l = x^{l+1}+1$.
The genus of $X$ is $g=l(l-1)/2$.
Let $\infty$ be the unique point on $X$ 
which does not lie above its affine part.
Fix a primitive $d$-th root of unity 
$\zeta_d \in \Q_p^*$,
and define an automorphism $\delta$ of $X$ by
$\delta(x, y) = (\zeta_d^{l}x, \zeta_d^{l+1} y)$. 
The order of $\delta$ is $d$.
Then we have the following for any $n \in \Z_{>0}$:
\begin{enumerate}
\item
It holds
$\Jac(X)[p^n]_1 \cap \Theta  = \{ 0 \}$.
\item
The formal power series
\[ l_1(X) = 
\sum_{k=0}^{\infty} 
\frac{1}{p^k} 
\begin{binom}
{\frac{p^k-1}{l}}
{\frac{p^k-1}{l+1}}
\end{binom}
\]
is the logarithm function of 
a formal group over $\Z_p$, which is isomorphic  to $\hat{J}_{X, 1}/\Z_p$.
\end{enumerate}
\end{example}

\begin{example}
Let $l$ be an odd prime such that $p \equiv 1 \mod d:=2l$.
Let $X$ be 
the smooth projective model of an affine curve 
over $\Q_p$ defined by 
$y^l  = x^{2a}(x^2+1)^b$, where 
$a, b$ are positive integers
such that $2(a+b)=l+1$.
The genus of $X$ is $g=l-1$.
Let $\infty$ be the unique point on $X$ 
which does not lie above its affine part.
Fix a primitive $l$-th root of unity 
$\zeta_l \in \Q_p^*$,
and define an automorphism $\delta$ of $X$ by
$\delta(x, y) = (-x, \zeta_l y)$. 
The order of $\delta$ is $d$.
Then we have the following for any $n \in \Z_{>0}$:
\begin{enumerate}
\item
It holds
$\Jac(X)[p^n]_1 \cap \Theta  = \{ 0 \}$.
\item
The formal power series
\[l_1(X) = 
\sum_{k=0}^{\infty} 
\frac{1}{p^k} 
\binom
{\frac{p^k-1}{l} b}
{\frac{p^k-1}{l} (2b-1)}
\]
is the logarithm function of 
a formal group over $\Z_p$, which is isomorphic  to $\hat{J}_{X, 1}/\Z_p$.
\end{enumerate}
\end{example}

\vspace{5mm}

\noindent
{\it Acknowledgement.} 
Part of this work was done while 
the authors were visiting 
National Center for Theoretical Sciences 
in Hsinchu, Taiwan.
We would like to thank Professor Yifan Yang and
the institute for their hospitality.

%
%
%
%
%

\end{document}